\documentclass{article}
\usepackage[utf8]{inputenc}

\usepackage[margin=1.2 in, top=1 in, bottom= 1.2 in]{geometry}

\usepackage{amsfonts}
\usepackage{mathrsfs}
\usepackage{bbm}
\usepackage{latexsym}
\usepackage{amssymb}
\usepackage{mathtools}
\usepackage[normalem]{ulem}

\usepackage{enumitem}
\setlist{nolistsep} 	
\usepackage{amsthm}
\usepackage{tikz-cd}

\usepackage{xcolor}

\usepackage[pdftex,pagebackref,colorlinks=true,urlcolor=blue,linkcolor=blue,citecolor=blue]{hyperref}
\usepackage[capitalize]{cleveref}

\usepackage{authblk}

\usepackage{lipsum}
\newcommand\blfootnote[1]{%
  \begingroup
  \renewcommand\thefootnote{}\footnote{#1}%
  \addtocounter{footnote}{-1}%
  \endgroup
}

\newtheoremstyle{plain}{3mm}{3mm}{\slshape}{}{\bfseries}{.}{.5em}{}
\newtheoremstyle{claim}{3mm}{3mm}{}{}{\itshape}{.}{.5em}{}
\newtheoremstyle{definition}{2mm}{2mm}{}{}{\bfseries}{.}{.5em}{}
\theoremstyle{plain}
	
\newtheorem{Theorem}{Theorem}[section]
\newtheorem{theorem}[Theorem]{Theorem}

\newtheorem{lemma}[Theorem]{Lemma}
\newtheorem{Proposition}[Theorem]{Proposition}
\newtheorem{proposition}[Theorem]{Proposition}

\newtheorem{Conjecture}[Theorem]{Conjecture}

\newtheorem{question}[Theorem]{Question}
\theoremstyle{claim}

\theoremstyle{definition}

\newtheorem{remark}[Theorem]{Remark}

\theoremstyle{plain} 
\newcounter{MainTheoremCounter}

\newtheorem{Maintheorem}[MainTheoremCounter]{Theorem}

\theoremstyle{plain}
\newtheorem*{namedthm}{\namedthmname}
\newcounter{namedthm}
\makeatletter
\newenvironment{named}[2]
{\def\namedthmname{#1}
\refstepcounter{namedthm}
\namedthm[#2]\def\@currentlabel{#1}}
{\endnamedthm}
\makeatother

\newcommand{\Szemeredi}{Szemer\'{e}di}

\newcommand{\N}{\mathbb{N}}
\newcommand{\Z}{\mathbb{Z}}
\newcommand{\R}{\mathbb{R}}
\newcommand{\C}{\mathbb{C}}

\newcommand{\T}{\mathbb{T}}

\newcommand{\ip}{\mathsf{IP}}
\newcommand{\Pol}{\operatorname{Pol}}

\renewcommand{\epsilon}{\varepsilon}
\renewcommand{\leq}{\leqslant}
\renewcommand{\geq}{\geqslant}
\renewcommand{\setminus}{\backslash}

\newcommand{\E}{\mathbb{E}}

\newcommand{\eps}{\epsilon}

\newcommand{\orb}{\overline{o}}

\newcommand{\defeq}{\coloneqq}
\newcommand{\family}{\mathscr{F}}

\title{A structure theorem for polynomial return-time sets in minimal systems}

\author[1]{Daniel Glasscock}
\author[2]{Andreas Koutsogiannis}
\author[3]{Anh N.\ Le}
\author[4]{Joel Moreira}
\author[5]{Florian K.\ Richter}
\author[6]{Donald Robertson}

\affil[1]{\small Dept.\ of Mathematics and Statistics, U. of Massachusetts Lowell, Lowell, MA, USA}
\affil[2]{\small Dept.\ of Mathematics, Aristotle University of Thessaloniki, Thessaloniki, Greece}
\affil[3]{\small Dept.\ of Mathematics,
	U. of Denver, Denver, CO, USA}
\affil[4]{\small Warwick Mathematics Institute, Coventry, UK}
\affil[5]{\small \'{E}cole Polytechnique F\'{e}d\'{e}rale de Lausanne, Lausanne, Switzerland}
\affil[6]{\small Dept.\ of Mathematics, University of Manchester, Manchester, UK}

\date{}

\begin{document}

\maketitle

\begin{abstract}
We investigate the structure of return-time sets determined by orbits along polynomial tuples in minimal topological dynamical systems.
Building on the topological characteristic factor theory of Glasner, Huang, Shao, Weiss, and Ye, we prove a structure theorem showing that, in a minimal system, return-time sets coincide -- up to a non-piecewise syndetic set -- with those in its maximal infinite-step pronilfactor.
As applications, we establish three new multiple recurrence theorems concerning linear recurrence along dynamically defined syndetic sets and polynomial recurrence along arithmetic progressions in minimal and totally minimal systems.
We also show how our main theorem can be used to prove that two previously separate conjectures -- one due to Glasner, Huang, Shao, Weiss, and Ye and the other due to Leibman -- are equivalent.
\end{abstract}

\blfootnote{2020 {\it Mathematics Subject Classification.}  Primary: 37B20. Secondary: 37B05.}

\blfootnote{{\it Key words and phrases. }
Minimal topological dynamical systems,
sets of recurrence,
polynomial multiple recurrence,
return-time sets,
nilsystems,
infinite-step pronilfactor.}

\tableofcontents

\section{Introduction}

Let $T: X \to X$ be a homeomorphism of a compact metric space $X$.
Given nonempty, open sets $U_1,\ldots,U_d \subseteq X$ and a polynomial tuple $p = (p_1, \ldots, p_d) \in \Z[x]^d$, the \emph{set of return-times of $U_1,\ldots,U_d$ along $p$} is defined as 
\begin{align}
\label{eqn_main_rp_notation}
    R_p(U_1,\dots,U_d):=\big\{n\in\Z: \ T^{-p_1(n)} U_1\cap\cdots\cap T^{-p_d(n)}U_d\neq\emptyset \big\}.
\end{align}
These sets naturally capture the polynomial recurrence behavior of the system $(X,T)$ by recording the times $n$ for which there exists a point $x\in X$ with $T^{p_i(n)}x \in U_i$ for every $i=1,\ldots,d$.
Return-time sets have been important objects of study since the pioneering work of Furstenberg and Weiss~\cite{Furstenberg_Weiss78}, who introduced topological dynamics as a powerful tool in additive combinatorics and Ramsey theory, linking dynamical recurrence to combinatorial and arithmetic structures in the integers.

In this paper, we expound on the principle that the sets $R_p(U_1,\ldots,U_d)$ are, to a large extent, determined by distinguished factors of $(X,T)$ with simpler and more rigid dynamics.
Factors that determine properties of return-time sets are known as \emph{characteristic} factors.
A major breakthrough -- due to Glasner, Huang, Shao, Weiss, and Ye \cite{glasner_huang_shao_weiss_ye_2020} in the linear case and extended by Qiu \cite{Qiu23} and Huang, Shao, and Ye \cite{huang_shao_ye_2023} in the polynomial case -- identifies the maximal $\infty$-step pro-nilfactor $(X_\infty,T)$ (see \cref{sec_prelims_nil}) as a characteristic factor.
A later refinement by Ye and Yu~\cite[Theorem~A]{yeyu} shows that the maximal $k$-step pronilfator is characteristic, where $k$ depends only on the polynomials $p_1,\ldots,p_d$.
These results provide long-awaited topological analogues of the existing ergodic-theoretic structure theory of measure-preserving systems established by Host, Kra \cite{Host_Kra05} and Ziegler \cite{Ziegler07} in the linear case and Host, Kra \cite{Host_Kra05b} and Leibman \cite{Leibman05} in the polynomial case.

\subsection{The main result}

Our main theorem -- \cref{mainthm_ps_structure_theory_pol} below -- is a refinement of the aforementioned topological characteristic factor results. It is best framed by formulating a statement equivalent to Qiu's theorem \cite[Theorem B]{Qiu23}; the equivalence between \cref{thm_qiu_eqiuvalent_statement} and Qiu's theorem is shown in \cref{thm_qiu_equivalent}.  A tuple of polynomials $(p_1, \ldots, p_d) \in \Z[x]^d$ is called {\it essentially distinct} if $p_i - p_j$ is non-constant whenever $i \neq j$. We denote by $U^\circ$ the interior of a subset $U$ of a topological space.

\begin{theorem}[{cf. \cite[Theorem B]{Qiu23} and \cref{thm_qiu_equivalent}}]
\label{thm_qiu_eqiuvalent_statement}
Let $(X,T)$ be a minimal and invertible topological dynamical system. Denote by $(X_\infty,T)$ its maximal $\infty$-step pro-nilfactor, and let $\pi\colon X\to X_\infty$ be the associated factor map.
For all $d \in \N$, all nonempty, open $U_1, \ldots, U_d \subseteq X$, and all essentially distinct $p \in \mathbb{Z}[x]^d$, if $R_p((\pi U_1)^\circ, \ldots, (\pi U_d)^\circ) \neq \emptyset$, then $R_p(U_1, \ldots, U_d) \neq \emptyset$.
\end{theorem}

\cref{thm_qiu_eqiuvalent_statement} says that recurrence among the sets $U_i$ along $p$ in a system is guaranteed by recurrence among the interiors of their images $(\pi U_i)^\circ$ along $p$ in the system's maximal $\infty$-step pronilfactor; in essence, recurrence along $p$ can be ``lifted'' from a system's $\infty$-step pronilfactor.  This is a surprising and highly useful converse to the fact that recurrence in a system implies recurrence in its factors, a simple consequence of the inclusion
\begin{align}
    \label{eqn_factor_recurrence_containment}
    R_p(U_1, \ldots, U_d) \subseteq R_p\big((\pi U_1)^\circ, \ldots, (\pi U_d)^\circ\big),
\end{align}
as shown in \cref{lemma_containment_in_one_direction}.

While \cref{thm_qiu_eqiuvalent_statement} can quickly be upgraded to show under the same hypotheses that the set $R_p(U_1, \ldots, U_d)$ is syndetic (see \cref{lem_nonempty_implies_syndetic}), it says nothing of how much larger $R_p\big((\pi U_1)^\circ, \ldots, (\pi U_d)^\circ\big)$ might be than $R_p(U_1, \ldots, U_d)$.  Our main result addresses this by showing that the two sets differ by a set that is small, in the sense that it is not piecewise syndetic.  Syndeticity and piecewise syndeticity are defined in \cref{sec_families}.

\begin{Maintheorem}
\label{mainthm_ps_structure_theory_pol}
Let $(X,T)$ be a minimal and invertible topological dynamical system. Denote by $(X_\infty,T)$ its maximal $\infty$-step pro-nilfactor, and let $\pi\colon X\to X_\infty$ be the associated factor map.
For all $d \in \N$, all nonempty, open $U_1, \ldots, U_d \subseteq X$, and all essentially distinct $p \in \mathbb{Z}[x]^d$, the difference set
\[
    R_p\big((\pi U_1)^\circ, \ldots, (\pi U_d)^\circ\big) \setminus R_p(U_1, \ldots, U_d)
\]
is not piecewise syndetic.
\end{Maintheorem}

We give examples in \cref{sec_equicont_factor} showing that the conclusion of \cref{mainthm_ps_structure_theory_pol} does not hold if $X_\infty$ is replaced by the system's maximal equicontinuous factor or if the essential distinctness assumption on $p$ is omitted. On the other hand, by the containment in \eqref{eqn_factor_recurrence_containment}, \cref{mainthm_ps_structure_theory_pol} is quickly seen to be equivalent to the ostensibly stronger version in which the factor map $\pi : X \to X_\infty$ is replaced by any factor map $\pi' : X \to Y$ to a larger factor $Y$, in the sense that there is a factor map $\eta : Y \to X_\infty$ with $\pi = \eta \circ \pi'$.

\cref{mainthm_ps_structure_theory_pol} is proved in \cref{sec_poly_ps_structure}.  The primary ingredients include a set-algebraic mechanism for proving non-piecewise syndeticity (see \cref{lemma_diff_not_ps}); the IP Polynomial \Szemeredi{} theorem of Bergelson and McCutcheon \cite[Lemma 6.12]{Bergelson_McCutcheon00}; and the polynomial topological characteristic factor theorem of Qiu \cite[Theorem~B]{Qiu23} mentioned above.

\subsection{Applications}
\label{sec_applications_intro}

\cref{mainthm_ps_structure_theory_pol} serves to reduce questions of multiple polynomial recurrence in minimal systems to ones in inverse limits of minimal nilsystems, a setting where recurrence is much better understood. We present four such applications below. All topological dynamical systems -- henceforth, just ``systems'' -- appearing in this work are invertible.

A set $A \subseteq \Z$ is a {\it set of multiple topological recurrence} if for all minimal systems $(X,T)$, all nonempty, open $U \subseteq X$, and all $d \in \N$, there exists $n \in A$ such that
\[
    U \cap T^{-n} U \cap \cdots \cap T^{-dn} U \neq \emptyset.
\]

We change ``multiple'' to ``single'' when $d = 1$, and we say that the set $A$ is a set of topological recurrence {\it for a family of systems} when the definition holds for all systems in the family.
Determining necessary and sufficient conditions for a set to be a set of topological recurrence has been an important objective, both for theory and applications.
That $\N$ is a set of multiple topological recurrence follows from Furstenberg and Weiss's topological multiple recurrence theorem \cite[Theorem~1.4]{Furstenberg_Weiss78}.

The two-point rotation shows that the set of odd numbers fails to be a set of single recurrence. 
This system is, however, the essential obstruction: in minimal systems disjoint from the two-point rotation (equivalently, in minimal systems $(X,T)$ for which $(X,T^2)$ is minimal), the set of odd numbers is a set of multiple recurrence.  
This was proved by Glasner, Huang, Shao, Weiss, and Ye, who established the following more general result.

\begin{theorem}[{\cite[Theorem D]{glasner_huang_shao_weiss_ye_2020}}]
\label{thm:odd-recurrence-linear-glasner}
Let $(X,T)$ be a system and $k, d \in \N$. 
If $(X,T^k)$ is minimal, then for all open $\emptyset\neq U \subseteq X$, the set 
\[\big\{ n \in \Z : \ U \cap T^{-n} U \cap \cdots \cap T^{-dn} U \neq \emptyset \big\}\]
has nonempty intersection with every infinite arithmetic progression of step size $k$.
\end{theorem}

In the following subsections, we show how \cref{mainthm_ps_structure_theory_pol} allows us to upgrade and extend \cref{thm:odd-recurrence-linear-glasner} to dynamically-defined syndetic sets (\cref{Mainthm:odd-recurrence-pol}), to polynomial multiple recurrence (\cref{cor:pol_k_p}), and to totally minimal systems (\cref{cor:totally_minimal}).
We conclude by showing how \cref{mainthm_ps_structure_theory_pol} implies that a generalization of \cref{thm:odd-recurrence-linear-glasner} conjectured by Glasner, Huang, Shao, Weiss, and Ye is equivalent to a long-standing open conjecture of Leibman (\cref{mainthm_conjecture_are_equiv}). 

\subsubsection{Linear recurrence along dynamically syndetic sets}
\label{sec_linear_rec_along_dy_synd_sets}

Our first application extends \cref{thm:odd-recurrence-linear-glasner} by replacing arithmetic progressions with sets of visit times in minimal systems satisfying a necessary disjointness condition. Given a minimal system $(Y,S)$, a point $y \in Y$, and a nonempty, open set $V \subseteq Y$, we write
\[R_S(y,V) \defeq \big\{ n \in \Z \ \big| \ S^n y \in V \big\}\]
for the set of times that $y$ visits $V$.
We call a subset of the integers {\it dynamically syndetic} if it contains a set of the form $R_S(y,V)$. Due to minimality, every dynamically syndetic set is syndetic; the converse is false: the set
\[
    \{n \in \Z: n \leq 0\} \cup \bigg( 2 \N \cap \bigcup_{\substack{k = 0 \\ \text{$k$ even}}}^\infty \big[ 2^{k}, 2^{k + 1}\big) \bigg) \cup \bigg( \big( 2\N + 1 \big) \cap \bigcup_{\substack{k = 0 \\ \text{$k$ odd}}}^\infty \big[ 2^{k}, 2^{k + 1}\big) \bigg).
\]
is syndetic, but not dynamically syndetic. For a proof, see \cite{glasscock_le_2024}.

\begin{remark}
    When $y \in V$, the set $R_S(y,V)$ is a set of polynomial multiple topological recurrence. Indeed, by IP Polynomial \Szemeredi{} theorem of Bergelson and McCutcheon \cite[Lemma 6.12]{Bergelson_McCutcheon00}, the set $R_p(U,\ldots,U)$ has non-empty intersection with every $\ip$ set; in particular, it has nonempty intersection with every central set. When $y \in V$, the set $R_S(y,V)$ is central, and hence has nonempty intersection with $R_p(U,\ldots,U)$. For the definitions of $\ip$ and central sets, see eg. \cite[Ch. 8.3]{Furstenberg81}.
\end{remark}

We show in the following theorem that all sets of the form $R_S(y,V)$ -- not just those for which $y \in V$ -- are sets of multiple linear topological recurrence for all systems $(X,T)$ that are sufficiently disjoint from $(Y,S)$. See \cref{sec_prelims_top_dyn} for the requisite definitions.

\begin{Maintheorem}
\label{Mainthm:odd-recurrence-pol}
Let $(X, T)$ and $(Y, S)$ be minimal systems that do not share any nontrivial eigenvalues.
For all nonempty, open sets $U \subseteq X$ and $V \subseteq Y$, all $y \in Y$ (not necessarily in $V$), and all $d \in \N$, the set
\begin{align}
\label{eqn_set_to_show_dcs}
    \big\{ n \in \Z: \ U \cap T^{-n} U \cap \cdots \cap T^{-dn} U \neq \emptyset \big\} \cap R_S(y,V)
\end{align}
is dynamically syndetic.
\end{Maintheorem}

In the case that $(Y,S)$ is a rotation of $k$ points, the condition that $(X,T)$ and $(Y,S)$ share no common eigenvalues is readily seen to be equivalent to the condition that the system $(X,T^k)$ is minimal. 
In this way, \cref{Mainthm:odd-recurrence-pol} generalizes \cref{thm:odd-recurrence-linear-glasner} from infinite progressions to other dynamically syndetic sets $R_S(y, V)$, including Beatty sequences $\{\lfloor n \alpha \rfloor: n \in \Z\}$ (arising from irrational rotations); $\{n \in \Z : \| n^2 \alpha \| < \epsilon \}$ (arising from nilsystems);  $\{n \in \Z : \nu_2(n) \text{ is even}\}$, where $\nu_2(n)$ denotes the $2$-adic valuation of an integer $n$ (arising from an almost 1--1 extension of an odometer); and so on.

The assumption that the systems $(X,T)$ and $(Y,S)$ have no common eigenvalues is necessary. 
Indeed, if $(X,T)$ and $(Y,S)$ share a common eigenvalue, then by \cite[Theorem 5]{Peleg72} (see also \cite{MR467704}), the product system $(X \times Y, T \times S)$ is not transitive: there exist nonempty, open sets $U_1, U_2 \subseteq X$ and $V_1, V_2 \subseteq Y$ such that for all $n \in \Z$, $(U_1 \times V_1) \cap (T \times S)^{-n} (U_2 \times V_2) = \emptyset$.  Let $n \in \Z$ be such that $U \defeq U_1 \cap T^{-n} U_2 \neq \emptyset$.  We see that for all $n \in \Z$, $(U \times V_1) \cap (T \times S)^{-n} (U \times V_2) = \emptyset$, whereby
\[\big\{ n \in \Z: \ U \cap T^{-n} U \neq \emptyset \big\} \cap \big\{ n \in \Z: \ V_1 \cap T^{-n} V_2 \neq \emptyset \big\} = \emptyset.\]
This contradicts the conclusion in \cref{Mainthm:odd-recurrence-pol} for any point $y \in V_1$, since for such points, $R_S(y,V_2) \subseteq \{n \in \Z : V_1 \cap T^{-n} V_2\}$.

The proof of \cref{Mainthm:odd-recurrence-pol} is completed in \cref{sec_proof_of_thm_b}.  We first use \cref{mainthm_ps_structure_theory_pol} to reduce it to the case that $(X,T)$ is a minimal nilsystem.  In that case, we show that the set in \eqref{eqn_set_to_show_dcs} is dynamically syndetic because it contains the set of times that a point $(x, \ldots, x, y) \in X^{d+1} \times Y$ visits the set $U \times \cdots \times U \times V$ under the map $I \times T \times T^2 \times \cdots \times T^d \times S$. The existence of such a point follows by \cref{lem:same-mef}, which guarantees that a generic diagonal point in $X^{d+1}$ generates a system that is disjoint from $(Y,S)$.

\subsubsection{Polynomial recurrence along arithmetic progressions}

Our second application of \cref{mainthm_ps_structure_theory_pol} is a generalization of \cref{thm:odd-recurrence-linear-glasner} to   polynomial multiple recurrence with dynamically syndetic intersections.  We say that two sets have {\it dynamically syndetic intersection} if their intersection is a dynamically syndetic set.

\begin{Maintheorem}
\label{cor:pol_k_p}
Let $(X,T)$ be a system and $k, d \in \N$. If $(X,T^k)$ is minimal, then for all polynomials $p \in \Z[x]$ with $p(0) = 0$ and all nonempty, open $U \subseteq X$, the set
\[\big\{ n \in \Z: \ U \cap T^{-p(n)} U \cap \cdots \cap T^{-dp(n)} U \neq \emptyset \big\}\]
has dynamically syndetic intersection with every infinite arithmetic progression of step size $k$.
\end{Maintheorem}

We prove \cref{cor:pol_k_p} in \cref{sec_proof_of_theorem_b} by first establishing the result for profinite nilsystems, then using \cref{mainthm_ps_structure_theory_pol} to lift the result to arbitrary minimal systems.

\subsubsection{Recurrence in totally minimal systems}

For our third application of \cref{mainthm_ps_structure_theory_pol}, we show that the ``$p(0) = 0$'' assumption in \cref{cor:pol_k_p} can be omitted in totally minimal systems.  This improves on theorems of Glasner, Huang, Shao, Weiss, and Ye \cite[Cor. 6.3]{glasner_huang_shao_weiss_ye_2020} and Qiu \cite[cf. Theorem A]{Qiu23}, who show the $d=1$ case for quadratic polynomials and non-constant polynomials, respectively.

\begin{Maintheorem}
\label{cor:totally_minimal}
    Let $(X,T)$ be totally minimal and $d \in \N$. For all non-constant  polynomials $p \in \Z[x]$ and all nonempty, open $U\subseteq X$, the set
    \begin{align}
    \label{eq_totallyminimalpolynomialrecurrence}
       \{n \in \Z: \ U\cap T^{-p(n)}U\cap\cdots\cap T^{-dp(n)}U\neq\emptyset\}
    \end{align}
    is dynamically syndetic. 
\end{Maintheorem}

Note that \cref{cor:totally_minimal} does not follow from \cref{cor:pol_k_p} by writing $p$ as some other polynomial along an arithmetic progression.  Indeed, not every $p \in \Z[x]$ can be written as $p(n) = q(kn + j)$ for some $q \in \Z[x]$ with $q(0) = 0$.  To see why, note that if $p(n) = q(kn + j)$ holds for all integers $n$, then it holds for all real numbers $n$.  In particular, setting $n = -j/k$, we see that $p(-j/k) = q(0) = 0$, whereby $p$ has a rational root.  Therefore, any integer-coefficient polynomial without rational roots cannot be written in this way.

We prove \cref{cor:totally_minimal} in \cref{sec_proof_of_thm_c} by using \cref{mainthm_ps_structure_theory_pol} to reduce to the case of minimal nilsystems, where we can appeal to the same set of results behind the proofs of Theorems \ref{Mainthm:odd-recurrence-pol} and \ref{cor:pol_k_p}.

\subsubsection{Conjectures of Glasner-Huang-Shao-Weiss-Ye and Leibman}

\cref{cor:pol_k_p} takes a step toward a positive resolution of the following generalization of \cref{thm:odd-recurrence-linear-glasner} conjectured by Glasner, Huang, Shao, Weiss, and Ye. We state it here in an equivalent form.

\begin{Conjecture}
[{\cite[Conjecture 3]{glasner_huang_shao_weiss_ye_2020}}]

\label{conj:polynomial-odd-recurrence_intro}
\label{conj:polynomial-odd-recurrence}
    Let $(X,T)$ be a minimal system and $k, d \in \N$. If $(X,T^k)$ is minimal, then for all polynomials $p_1,\dots,p_d \in \Z[x]$ with $p_i(0)=0$ for $i = 1, \ldots, d$, and all nonempty, open $U \subseteq X$, the set
    \[\big\{ n \in \Z : \ T^{-p_1(n)}U\cap\cdots\cap T^{-p_d(n)}U \neq \emptyset \big\}\]
    has nonempty intersection with all infinite arithmetic progressions of step size $k$.
\end{Conjecture}

In light of \cref{mainthm_ps_structure_theory_pol}, to prove \cref{conj:polynomial-odd-recurrence_intro}, it suffices to prove the same statement under the assumption that $(X,T)$ is a nilsystem.
Polynomial orbits in nilsystems have been extensively studied. Most relevant to the present work is a conjecture of Leibman \cite[Conjecture 11.4]{Leibman10b}, which gives an explicit description of the orbit closure of the diagonal $\Delta_{X^d} = \{(x, \ldots, x) : x \in X\}$ in a connected nilsystem $(X = G/\Gamma, a)$ under a polynomial sequence.
The following conjecture, while formally a special case of \cite[Conjecture 11.4]{Leibman10b}, embodies what we believe to be the central idea that Leibman’s conjecture seeks to convey.
Notation for nilsystems and nilrotations is found in \cref{sec_prelims_nil}.

\begin{Conjecture}[See {\cite[Conjecture 11.4]{Leibman10b}}]
\label{conj:leibmanish_intro}
\label{conj:leibmanish}
Let $(X = G/\Gamma, a)$ be a totally minimal nilsystem. Let $p_1, \ldots, p_d \in \Z[x]$ with $p_i(0) = 0$, and let $g(n) = a^{p_1(n)} \otimes \cdots \otimes a^{p_d(n)}$ be the corresponding polynomial sequence in $G^{d}$. The set
\[
    \overline{\{g(n) \Delta_{X^d}: \ n \in \Z\}}
\]
is connected.
\end{Conjecture}

Our main result in this section is that these conjectures of Glasner, Huang, Shao, Weiss and Ye and Leibman are, in fact, equivalent. The proof of \cref{mainthm_conjecture_are_equiv} is completed in \cref{sec:all_equivalent_conjectures}.

\begin{Maintheorem}
\label{mainthm_conjecture_are_equiv}
Conjectures \ref{conj:polynomial-odd-recurrence_intro} and \ref{conj:leibmanish_intro} are equivalent. 
\end{Maintheorem}

It is interesting to juxtopose the minimality assumption in \cref{conj:polynomial-odd-recurrence} with the total minimality assumption in \cref{conj:leibmanish_intro}. Many recurrence problems related to \cref{conj:polynomial-odd-recurrence_intro} become much simpler when total minimality is assumed. In fact, it is stated explicitly in \cite{glasner_huang_shao_weiss_ye_2020} that \cref{conj:polynomial-odd-recurrence_intro} may be easier if we assume $(X, T)$ is totally minimal. The equivalence of this conjecture and \cref{conj:leibmanish_intro} demonstrates that, unexpectedly, assuming total minimality does not make \cref{conj:polynomial-odd-recurrence_intro} easier.  We make this fact precise below by formulating \cref{conj:glasner-etal-totally-minimal} -- a version of \cref{conj:polynomial-odd-recurrence_intro} for totally minimal nilsystems -- and showing that it, too, is equivalent to \cref{conj:polynomial-odd-recurrence_intro}.

We conclude by remarking that \cref{conj:polynomial-odd-recurrence_intro} is equivalent to its ergodic-theoretic analogue, in which minimal systems and open sets are replaced by ergodic probability measure preserving systems and sets of positive measure.

\begin{Conjecture}
\label{conj_meas_analogue_of_poly_odd_rec}
    Let $(X,\mu,T)$ be an ergodic measure preserving system and $k, d \in \N$. If $(X,\mu,T^k)$ is ergodic, then for all polynomials $p_1,\dots,p_d$ over $\Z$ with $p_i(0)=0$ for $i = 1, \ldots, d$, and all measurable $A \subseteq X$ with $\mu(A) > 0$, the set
    \[\big\{ n \in \Z : \ \mu(T^{-p_1(n)}A\cap\cdots\cap T^{-p_d(n)}A) > 0 \big\}\]
    has nonempty intersection with all infinite arithmetic progressions of step size $k$.
\end{Conjecture}

Although we will not develop this here, the equivalence between Conjectures \ref{conj:polynomial-odd-recurrence_intro} and \ref{conj_meas_analogue_of_poly_odd_rec} can be seen by using the structure theory of measure preserving systems to reduce \cref{conj_meas_analogue_of_poly_odd_rec} from arbitrary ergodic systems to ergodic nilsystems. 
This implies that \cref{conj_meas_analogue_of_poly_odd_rec} is equivalent to \cref{lem:connected_pol_first} below, which we show in \cref{sec:all_equivalent_conjectures} is also equivalent to \cref{conj:polynomial-odd-recurrence_intro}.

\subsection{Outline}

After covering necessary preliminaries in \cref{sec_prelims}, we prove \cref{mainthm_ps_structure_theory_pol} in \cref{sec_poly_ps_structure}.  Applications -- including Theorems \ref{Mainthm:odd-recurrence-pol}, \ref{cor:pol_k_p}, \ref{cor:totally_minimal}, and \ref{mainthm_conjecture_are_equiv} -- are shown in \cref{sec_applications}.  We conclude with some open questions and directions in \cref{sec_further}.

\subsection{Acknowledgments}
The authors gratefully acknowledge the support of the American Institute of Mathematics, who supported this work under an AIM SQuaREs grant.
The second author was supported by the H.F.R.I call ``3rd Call for H.F.R.I.’s Research Projects to Support Faculty Members \& Researchers'' (H.F.R.I. Project Number: 24979).
The fourth author was supported by the EPSRC Frontier Research Guarantee grant EP/Y014030/1.
The fifth author was supported by the Swiss National Science Foundation grant TMSGI2-211214.

\section{Preliminaries}
\label{sec_prelims}

We denote by $\N$ and $\Z$ the set of all positive integers and integers, respectively.

\subsection{Families of subsets of \texorpdfstring{$\Z$}{Z}}
\label{sec_families}

Let $\family$ be a collection of subsets of $\Z$.  The \emph{upward closure} of $\family$, denoted $\uparrow \family$, is the set of all subsets of $\Z$ that contain at least one member of $\family$.  A collection of subsets of $\Z$ is \emph{upward closed} if it is equal to its upward closure.  For $A \subseteq \Z$, we define
\[
    A - \family = \{n \in \Z: \ A - n \in \family\},
\]
that is, the set $A - \family$ consists of those elements of $\Z$ which translate $A$ into $\family$.

Each of the following notions leads to an upward-closed collection of subsets of $\Z$ important to the intersection of Ramsey Theory and topological dynamics.  A set $A \subseteq \Z$ is \dots
\begin{itemize}
    \item \emph{syndetic} if there exists $N \in \N$ such that
    \[
        A \cup (A-1) \cup \cdots \cup (A-N) = \Z;
    \]

    \item \emph{thick} if for all $N \in \N$, there exists $z \in \Z$ such that $\{z, z+1, \ldots, z+N\} \subseteq A$; and

    \item \emph{piecewise syndetic} if it is the intersection of a syndetic set and a thick set, or, equivalently, there exists $N \in \N$ such that the set
    \[
        A \cup (A-1) \cup \cdots \cup (A-N) \text{ is thick}.
    \]
\end{itemize}

The following elementary lemma plays a crucial role in our proof of \cref{mainthm_ps_structure_theory_pol}.

\begin{lemma}
\label{lemma_diff_not_ps}
    Let $A, B \subseteq \Z$ and let $\family$ be an upward-closed family of subsets of $\Z$ with the property that the intersection of all sets in any finite sub-collection of $\family$ is syndetic. If $B \subseteq A \subseteq B - \family$, then $A \setminus B$ is not piecewise syndetic.
\end{lemma}

\begin{proof}
    Suppose, for the sake of a contradiction, that $A\setminus B$ is piecewise syndetic, and let $\ell\in\N$ be such that $T:=\bigcup_{i=1}^\ell (A\setminus B)-i$ is thick.
    Let $F\subseteq A$ be the largest subset of $A$ that is contained in an interval of length $\ell$, and let $I\supseteq F$ be such an interval. 
    Since $A \subseteq B-\family$, we have that $B-f \in \family$ for all $f \in F$.  
    By the assumption on $\family$, the intersection $S:=\bigcap_{f \in F} (B-f)$ is syndetic.
    Since $T$ is thick, so is any of its shifts and hence we can find $n\in S\cap \big(\bigcup_{i\in I}(A\setminus B)-i\big)$.
    From $n\in S$ it follows that $n+F\subseteq B\subseteq A$.
    From $n\in \bigcup_{i\in I}(A\setminus B)-i$ it follows that $n+I$ has an element $x\in A\setminus B$; in particular $x\notin n+F$.
    Therefore, $(n+I)\cap A\supseteq\{x\}\cup(n+F)$ whence it follows that $|(n+I)\cap A|>|F|$, contradicting the construction of $F$.
\end{proof}

The following lemma, when combined with \cref{mainthm_ps_structure_theory_pol}, allows us to lift recurrence from the maximal infinite step pronilfactor of a system. Recall the definition of a dynamically syndetic set from \cref{sec_linear_rec_along_dy_synd_sets}.

\begin{lemma}
\label{lemma_pwssetdifference}
    Let $S, B\subseteq\Z$. If $S$ is (dynamically) syndetic and $S \setminus B$ is not piecewise syndetic, then $B$ is (dynamically) syndetic.
\end{lemma}
\begin{proof}
    Suppose $S$ is syndetic and $S \setminus B$ is not piecewise syndetic. To show $B$ is syndetic, we will show that $B$ has nonempty intersection with every thick set.
    Let $T\subseteq\Z$ be thick.
    Since $B\cap T\supseteq(S\cap T)\setminus (S\setminus B)$ and $S\cap T$ is piecewise syndetic (but $S\setminus B$ is not), we deduce that $B\cap T\neq\emptyset$, as desired.  

    If $S$ is dynamically syndetic, the conclusion of the lemma follows from \cite[Theorem G]{glasscock_le_2025}.
\end{proof}

\subsection{Topological dynamics}
\label{sec_prelims_top_dyn}

An \emph{(invertible) topological dynamical system} is a compact metric space $X$ together with a homeomorphism $T: X \to X$. We frequently refer to invertible topological dynamical systems simply as \emph{systems} in this paper.  A set $W \subseteq X$ is \emph{$T$-invariant} if $T^n W \subseteq W$ for all $n \in \Z$.  If $W \subseteq X$ is nonempty, closed, and $T$-invariant, the system $(W,T)$ is called a \emph{subsystem} of $(X,T)$.

Given a subset $V\subseteq X$, the set $\orb_T(V)=\overline{\bigcup_{n\in\Z}\{T^nx:x\in V\}}$ is the \emph{orbit closure of $V$ under $T$}. When $V=\{x\}$ is a singleton, then we write $\orb_T(x)$ instead of $\orb_T(\{x\})$ and refer to it as the \emph{orbit closure of $x$ under $T$}. A system $(X,T)$ is \emph{transitive} if at least one point has a dense orbit (equivalently, for all nonempty, open $U, V \subseteq X$, there exists $n \in \Z$ such that $U \cap T^{-n} V \neq \emptyset$); \emph{minimal} if every point has a dense orbit; and \emph{totally minimal} if for all $n \in \N$, the system $(X,T^n)$ is minimal.

A \emph{factor} of a system $(X,T)$ is another system $(Y,S)$ together with a continuous surjection $\pi\colon X\to Y$, called the \emph{factor map} from $(X,T)$ onto $(Y,S)$, such that
\[
\pi\circ T=S\circ \pi.
\]
If $(Y,S)$ and $(Z,R)$ are two factors of $(X,T)$ with corresponding factor maps $\pi\colon X\to Y$ and $\eta\colon X\to Z$, then we say that $(Y,S)$ \emph{contains} $(Z,R)$ if there exists a factor map $\psi \colon Y \to Z$ with $\eta = \psi \circ \pi$.
When $(Y,S)$ is a factor of $(X,T)$, we equivalently call $(X,T)$ an \emph{extension} of $(Y,S)$.  An extension $\pi: X \to Y$ is \emph{almost 1--1} if there exists a residual set $\Omega \subseteq X$ such that for all $x \in \Omega$, $\pi^{-1}( \pi x) = \{x\}$.

A map $\phi : X \to Y$ between two topological spaces $X$ and $Y$ is \emph{semiopen} if for all nonempty, open $U \subseteq X$, the set $\phi U$ has nonempty interior. We use the following fact several times throughout the paper; the proof is contained in the proof of \cite[Lemma 2.9]{glasscock_koutsogiannis_richter_2019}.

\begin{lemma}
\label{lemma_semiopen}
    A factor map between two minimal systems is semiopen.
\end{lemma}

A system $(X,T)$ is \emph{equicontinuous} if the family of maps $\{T^n : n \in \Z\}$ is equicontinuous.  Every system $(X,T)$ has a \emph{maximal equicontinuous factor} (see, eg., \cite[Ch. 9]{Auslander88}); it is an equicontinuous system that is maximal in the sense that it contains every equicontinuous factor of $(X,T)$.

A number $\theta \in \R/\Z$ is an \emph{eigenvalue} of the system $(X,T)$ if there exists a continuous function $f: X \to \C$ for which $f \circ T = e^{2\pi i \theta}f$.  
Note that $0$ is an eigenvalue of all systems and hence it is called \emph{trivial}. Any other eigenvalue is called \emph{non-trivial}.
In a transitive system, every eigenfunction $f\in C(X)$ has a constant absolute value.
When $\theta$ is irrational, the function $f$ can be seen (after normalizing) as a factor map $f:X\to S^1$ to the unit circle under rotation by angle $2\pi\theta$. As this is an equicontinuous factor of $(X,T)$, it is contained in the maximal equicontinuous factor, and thus the set of eigenvalues of a transitive system coincides with those of its maximal equicontinuous factor.
The same is true when $\theta$ is rational if one replaces $S^1$ with one if its closed subgroups.

\subsection{Nilsystems and pro-nilsystems}
\label{sec_prelims_nil}

A \emph{$k$-step nilsystem} is a system $(G/\Gamma,T)$ where $G$ is a $k$-step nilpotent Lie group, $\Gamma\leq G$ is a discrete cocompact subgroup, and $T:x\mapsto ax$ is left translation on $X$ by some fixed $a\in G$.
There is a unique Borel probability on $G/\Gamma$ call the \emph{Haar measure} that is invariant under the action of $G$ on $G/\Gamma$. Every orbit closure in a $k$-step nilsystem is uniquely ergodic, and $k$-step nilsystems are always distal. For these and other standard facts about nilsystems, we refer to \cite{Host_Kra18}.

An inverse limit of $k$-step nilsystems is called a \emph{$k$-step pronilsystem}. Every minimal system $(X,T)$ admits a unique factor that is a $k$-step pronilsystem and contains all $k$-step nilfactors of $(X,T)$ as a factor (see~\cite[Chapter 17]{Host_Kra18}). It is called the \emph{maximal $k$‑step pronilfactor} of $(X,T)$, and we denote it by $(X_k,T)$. Note that $(X_1,T)$ is the maximal equicontinuous factor of $(X,T)$.  As $k$ tends to infinity, the maximal $k$‑step pronilfactors of $(X,T)$ form an increasing tower of factors. The inverse limit
\[
(X_\infty,T):=\varprojlim_{k\to\infty} (X_k,T)
\]
is the \emph{maximal $\infty$‑step pronilfactor} of $(X,T)$.

\begin{remark}
\label{rmk_open_set_in_an_inverse_limit}
As an inverse limit, the topology on $X_\infty$ is determined by the associated projection maps $\pi_k : X_\infty \to X_k$. Thus, if $U \subseteq X_\infty$ is open, there is $k \in \N$ and $V \subseteq X_k$ open with $\pi_k^{-1}(V) \subseteq U$.
\end{remark}

We will encounter polynomial sequences is nilpotent Lie groups.  Generally, a \emph{polynomial sequence} in a group $G$ is any sequence of the form
\[
g(n) = a_1^{p_1(n)} \cdots a_d^{p_d(n)}
\]
where $a_1,\dots,a_d$ are fixed elements of $G$ and $p_1,\dots,p_d$ are fixed polynomials in $\Z[x]$.

\section{A structure theorem for return-time sets}
\label{sec_poly_ps_structure}

Throughout this section -- the goal of which is to prove \cref{mainthm_ps_structure_theory_pol} -- and the next, it will be useful to recall the definition of the set $R_p(U_1,\ldots,U_d)$ from \eqref{eqn_main_rp_notation} in the introduction.  Recall also that a tuple of polynomials $p = (p_1, \ldots, p_d) \in \Z[x]^d$ is essentially distinct if for all $1 \leq i < j \leq d$, the polynomial $p_i - p_j$ is non-constant.  We will need to consider the following subfamily of essentially distinct tuples:
\[\Pol_d:=\Big\{(p_1,\dots,p_d)\in\Z[x]^{d}: \ p_i(0)=0, \ p_i\neq p_j\Big\}.\]

\subsection{Basics}

The following lemmas will be used several times in this section.
\begin{lemma}
\label{lem_translated_poly_returns}
    Let $(X,T)$ be a system, $U_1, \ldots, U_d \subseteq X$ be nonempty and open, $p=(p_1,\dots,p_d) \in \Z[x]^d$ be essentially distinct, and $a \in \Z$.  For each $i \in \{1, \ldots, d\}$, define $q_i(n) =  p_i(n + a) - p_i(a)$.  The tuple $q:=(q_1,\dots,q_d) \in \Pol_d$ and
    \[
        R_q(T^{-p_1(a)}U_1, \ldots, T^{-p_d(a)}U_d) = R_p(U_1, \ldots, U_d) - a.
    \]
\end{lemma}

\begin{proof}
Clearly $q_i(0) = 0$, so to show that $q \in \Pol_d$, we need only to show that $q_i \neq q_j$ for $i \neq j$.  
If $q_i=q_j$, then $n \mapsto p_i(n+a) - p_j(n+a)$ is constant. 
This implies that $p_i-p_j$ is constant, a contradiction if $i \neq j$.

The equivalences
\begin{eqnarray*}
n \in R_p(U_1, \ldots, U_d)-a 
 &\Longleftrightarrow& n+a \in R_p(U_1, \ldots, U_d)\\
 &\Longleftrightarrow& T^{-p_1(n+a)}U_1\cap\ldots\cap T^{-p_d(n+a)}U_d\neq \emptyset\\
 &\Longleftrightarrow& T^{-q_1(n)}(T^{-p_1(a)} U_1)\cap\ldots\cap T^{-q_d(n)}(T^{-p_d(a)} U_d)\neq \emptyset\\
 &\Longleftrightarrow& n\in R_q(T^{-p_1(a)}U_1,\ldots, T^{-p_d(a)}U_d)
\end{eqnarray*}
imply the second conclusion.
\end{proof}

\begin{lemma}
\label{lemma_containment_in_one_direction}
    Let $(X,T)$ be a minimal system and $\pi: (X,T) \to (Y,T)$ be a factor map.  For all nonempty, open $U_1, \ldots, U_d \subseteq X$ and all $p \in \Z[x]^d$,
    \[R_p(U_1,\ldots,U_d) \subseteq R_p\big((\pi U_1)^\circ, \ldots, (\pi U_d)^\circ\big).\]
\end{lemma}

\begin{proof}
    Let $U_1, \ldots, U_d \subseteq X$ be nonempty and open, $p = (p_1, \ldots, p_d) \in \Z[x]^d$, and $n \in R_p(U_1, \ldots, U_d)$.  The set $T^{-p_1(n)}U_1 \cap \cdots \cap T^{-p_d(n)}U_d$ is open and nonempty.
    Since $\pi$ is semiopen (\cref{lemma_semiopen}) and commutes with $T^{-1}$, the set
    \begin{align*}
        T^{-p_1(n)} \pi U_1 \cap \cdots \cap T^{-p_d(n)}\pi U_d
        \supseteq \pi \big( T^{-p_1(n)}U_1 \cap \cdots \cap T^{-p_d(n)}U_d \big)
    \end{align*}
    has nonempty interior.  It follows that
    \[
        (T^{-p_1(n)} \pi U_1)^\circ \cap \cdots \cap (T^{-p_d(n)}\pi U_d)^\circ = (T^{-p_1(n)} \pi U_1 \cap \cdots \cap T^{-p_d(n)}\pi U_d)^\circ \neq \emptyset.
    \]
    Since $T$ is a homeomorphism,
    \[
        \big(T^{-p_i(n)} \pi U_i \big)^{\circ} = T^{-p_i(n)} \big( \pi U_i \big)^{\circ}.
    \]
    As a result,
    \[
        T^{-p_1(n)}\big(\pi U_1 \big)^\circ \cap \cdots \cap T^{-p_d(n)}\big(\pi U_d \big)^\circ \neq \emptyset,
    \]
    whereby $n \in R_p\big( (\pi U_1)^\circ, \ldots, (\pi U_d)^\circ \big)$, as was to be shown.
\end{proof}

\subsection{Reduction to open extensions}
\label{sec_reduction_to_open}

We will show in this subsection that in order to prove \cref{mainthm_ps_structure_theory_pol}, it suffices to prove it under the assumption that the factor map $\pi$ is open. First, we need a lemma.

\begin{lemma}
\label{lemma_multiple_set_returns_lift_in_almost_11_extensions}
Let $\pi: (X,T) \to (Y,T)$ be an almost 1--1 extension of minimal systems.  For all nonempty, open $U_1, \ldots, U_d \subseteq X$ and all tuples of polynomials $p \in \Z[x]^d$,
\[
R_p(U_1,\ldots,U_d) = R_p\big((\pi U_1)^\circ, \ldots, (\pi U_d)^\circ\big).
\]
\end{lemma}

\begin{proof}
    The containment
    \[R_p(U_1,\ldots,U_d) \subseteq R_p\big((\pi U_1)^\circ, \ldots, (\pi U_d)^\circ\big)\]
    follows from \cref{lemma_containment_in_one_direction}. To see the reverse containment, suppose that $n \in R_p ( (\pi U_1)^\circ, \ldots, (\pi U_d)^\circ)$ so that
    \[\emptyset \neq T^{-p_1(n)} (\pi U_1)^\circ \cap \cdots \cap T^{-p_d(n)} (\pi U_d)^\circ = (T^{-p_1(n)} \pi U_1)^\circ \cap \cdots \cap (T^{-p_d(n)} \pi U_d)^\circ.\]
    It follows that the set $T^{-p_1(n)} \pi U_1 \cap \cdots \cap T^{-p_d(n)} \pi U_d$ has nonempty interior.  Since $\pi$ is almost 1--1, there exists  $y \in T^{-p_1(n)}\pi U_1 \cap \cdots \cap T^{-p_d(n)}\pi U_d$ such that $\pi^{-1}\{y\}$ is a singleton, say $\{x\}$.  It follows that $x \in T^{-p_1(n)}U_1 \cap \cdots \cap T^{-p_d(n)}U_d$, whereby $n \in R_p(U_1,\ldots,U_d)$, as was to be shown.
\end{proof}

The following theorem, which we prove in \cref{sec:proof_a_open}, is a special case of \cref{mainthm_ps_structure_theory_pol} in which the factor map $\pi$ is open and the polynomials are assumed to have no constant terms.

\begin{theorem}
\label{thm_ps_structure_theory_open}
\label{conj_ps_structure_theory_open}
    Let $(X,T)$ be a minimal system, and let $\pi: (X,T) \to (Y,T)$ be a factor of $X$. Suppose that $\pi$ is open and that $Y$ contains the infinite step pronilfactor of $X$. For all nonempty, open $U_1, \ldots, U_d \subseteq X$ and all $p \in \Pol_d$, the set
    \[
        R_p(\pi U_1, \ldots, \pi U_d) \setminus R_p(U_1, \ldots, U_d)
    \]
    is not piecewise syndetic.
\end{theorem}

Assuming \cref{conj_ps_structure_theory_open}, we now give a proof of \cref{mainthm_ps_structure_theory_pol}.

\begin{proof}[Proof of \cref{mainthm_ps_structure_theory_pol} using \cref{conj_ps_structure_theory_open}]
    Let $(X,T)$ be a minimal system, and let $\pi: X \to X_\infty$ be the factor map to the $\infty$-step pronilfactor of $X$.
    
    Let $d \in \N$, $\tilde U_1, \ldots, \tilde U_d \subseteq X$ be open and nonempty, and $\tilde p = (\tilde p_1,\ldots,\tilde p_d) \in \Z[x]^d$ be essentially distinct.
    Our goal is to show that the set
    \begin{equation}\label{eq_proof_thmA_1}
        R_{\tilde p}\big((\pi \tilde U_1)^\circ,\dots,(\pi \tilde U_d)^\circ\big) \big\setminus R_{\tilde p}\big(\tilde U_1,\dots,\tilde U_d\big)
    \end{equation}
    is not piecewise syndetic.  Define $p_i(n) = \tilde p_i(n)-\tilde p_i(0)$ and $U_i = T^{-\tilde p_i(0)}\tilde U_i$.
    Since each $T^{-\tilde p_i(0)}$ is a homeomorphism that commutes with $\pi$, it follows by \cref{lem_translated_poly_returns} that $p:=(p_1,\dots,p_d)\in\Pol_d$, 
    \begin{align*}
        R_p\big((\pi U_1)^\circ,\dots,(\pi U_d)^\circ\big) &= R_{\tilde p}\big((\pi \tilde U_1)^\circ,\dots,(\pi \tilde U_d)^\circ\big),
    \text{ and } \\
    R_p(U_1,\dots,U_d) &= R_{\tilde p}(\tilde U_1,\dots,\tilde U_d).
    \end{align*}
    Therefore we can express the set in \eqref{eq_proof_thmA_1} as
    \[
    R_p\big((\pi U_1)^\circ,\dots,(\pi U_d)^\circ \big)\setminus R_p(U_1,\dots,U_d).
    \]
    In other words, we may assume without loss of generality that $p\in\Pol_d$, rather than only that $p\in\Z[x]^d$ is essentially distinct.  This will allow us to apply \cref{conj_ps_structure_theory_open}.
    
    By the O-diagram construction (see \cite[Theorem~3.1]{Veech-pointDistalFlow} or \cite[Theorem~2.19]{glasner_huang_shao_weiss_ye_2020}), there exist an open extension of minimal systems $\pi^*: (X^*, T) \to (X_\infty^*, T)$ and almost 1--1 extensions $\tau: (X^*, T) \to (X,T)$ and $\sigma: (X_\infty^*, T) \to (X_\infty,T)$ such that $\pi \circ \tau = \sigma \circ \pi^*$:
    \begin{equation}
    \label{tikz_comm_diagram}
        \begin{tikzcd}[row sep = large, column sep = large]
            X \ar[r, leftarrow, "\tau"] \ar[d, rightarrow, "\pi"] &  X^* \ar[d, rightarrow, "\pi^*"]\\
            X_{\infty} \ar[r, leftarrow, "\sigma"] & X_{\infty}^*
        \end{tikzcd}
    \end{equation}
    It is the case that $X^*_\infty$ is the $\infty$-step pronilfactor of $X^*$ and that $\pi^*: X^* \to X^*_\infty$ is the associated factor map (see \cite[Lemma 5.6]{glasner_huang_shao_weiss_ye_2020}).

    Since $\sigma$ is an almost 1--1 extension of minimal systems, it follows by \cref{lemma_multiple_set_returns_lift_in_almost_11_extensions} that
    \[
        R_p(\pi^* \tau^{-1}U_1, \ldots, \pi^* \tau^{-1}U_d) = R_p((\sigma \pi^* \tau^{-1}U_1)^\circ, \ldots, (\sigma \pi^* \tau^{-1}U_d)^\circ).
    \]
    From $\pi \circ \tau = \sigma \circ \pi^*,$ we get that $\sigma \pi^* \tau^{-1}U_i = \pi U_i$.  
    Therefore, from the previous line,
    \[
        R_p(\pi^* \tau^{-1}U_1, \ldots, \pi^* \tau^{-1}U_d) = R_p((\pi U_1)^\circ, \ldots, (\pi U_d)^\circ).
    \]
    We see then that
    \[
        R_p(U_1, \ldots, U_d) = R_p(\tau^{-1}U_1, \ldots, \tau^{-1}U_d).
    \]
    By \cref{conj_ps_structure_theory_open} applied to $\pi^*$, the sets $\tau^{-1}U_1$, \dots, $\tau^{-1}U_d$, and the tuple $p$, we have that
    \begin{align}
    \label{eqn_applying_ps_structure_to_shadow_v2}
        R_p(\pi^* \tau^{-1}U_1, \ldots, \pi^* \tau^{-1}U_d) \setminus R_p(\tau^{-1}U_1, \ldots, \tau^{-1}U_d) \text{ is not piecewise syndetic.}
    \end{align}
    Therefore,
    \[
        R_p((\pi U_1)^\circ, \ldots, (\pi U_d)^\circ) \setminus R_p(U_1, \ldots, U_d)
    \]
    is not piecewise syndetic, as was to be shown.
\end{proof}

It remains to prove \cref{thm_ps_structure_theory_open}. We do this in \cref{sec:proof_a_open} following preparatory results in \cref{subsec:poly_visit_families} and \cref{subsec:bmq}.

\subsection{Families of polynomial visits}
\label{subsec:poly_visit_families}

The following lemma describes a property enjoyed by families of sets of visit times among sets whose images have non-empty intersection.

\begin{lemma}
\label{lemma_special_sets_of_returns_form_a_filterpol}
    Let $\pi: X \to Y$ be an open factor map of systems, and suppose $(Y,T)$ is transitive. For all $d, e \in \N$ and all $U_1,\ldots, U_d, V_1,\ldots, V_e \subseteq X$ nonempty and open satisfying $\pi U_1\cap\ldots\cap \pi U_d \neq \emptyset$ and $\pi V_1\cap\ldots\cap \pi V_e \neq \emptyset$, and for all $p \in\Pol_d$ and $q \in \Pol_e$, there exist $W_1, \ldots, W_{d+e} \subseteq X$ open, nonempty that satisfy $\pi W_1\cap\ldots\cap \pi W_{d+e} \neq \emptyset,$ and $r \in \Pol_{d+e}$ such that
    \[
        R_r(W_1,\ldots,W_{d+e}) \subseteq R_p(U_1,\ldots,U_d) \cap R_q(V_1,\ldots,V_e).
    \]
\end{lemma}

\begin{proof}
    Suppose $U_1,\ldots, U_d, V_1,\ldots, V_e \subseteq X$ are nonempty and open satisfying $\pi U_1\cap\ldots\cap \pi U_d \neq \emptyset$ and $\pi V_1\cap\ldots\cap \pi V_e \neq \emptyset$.  Since $\pi$ is open, the sets $\pi U_1\cap\ldots\cap \pi U_d$ and $\pi V_1\cap\ldots\cap \pi V_e$ are open and nonempty subsets of $Y$.  By transitivity, there exists $k \in \Z$ such that
    \[\pi U_1\cap\ldots\cap \pi U_d \cap T^{-k} \big(\pi V_1\cap\ldots\cap \pi V_e \big) \neq \emptyset.\]
    We see that $U_1, \ldots, U_d, T^{-k}V_1, \ldots, T^{-k}V_e$ are nonempty, open subsets of $X$.  Since $T^{-k} \pi = \pi T^{-k}$, we see
    \[\pi U_1\cap\ldots\cap \pi U_d \cap \pi T^{-k} V_1\cap\ldots\cap \pi T^{-k} V_e  \neq \emptyset.\]
    If $p=(p_1,\ldots,p_d)$ and $q=(q_1,\ldots,q_e),$ let $m \in \Z$ be such that the tuple
    \[r \defeq (p_1, \ldots, p_d, q_1 + mx, \ldots, q_e + mx) \in \Z[x]^{d+e}\]
    consists of pairwise distinct polynomials, that is, belongs to $\Pol_{d+e}$.
    We need only to note that the set $R_{r} (U_1,\ldots, U_d, T^{-k}V_1, \ldots, T^{-k}V_e)$ is a subset of
    \[
        R_p(U_1,\ldots,U_d) \cap R_q(V_1,\ldots,V_e),
    \]
    as was to be shown.
\end{proof}

\subsection{Theorems of Bergelson-McCutcheon and Qiu}
\label{subsec:bmq}

This subsection contains highly non-trivial inputs to the proof of our main theorem: the IP Polynomial \Szemeredi{} theorem of Bergelson and McCutcheon \cite[Lemma 6.12]{Bergelson_McCutcheon00} in \cref{lem_nonempty_implies_syndetic}, and a polynomial recurrence lifting theorem of Qiu \cite[Theorem~B]{Qiu23} in \cref{thm_QiuC}.

The following lemma shows that in minimal systems, if a return-time set is non-empty, then it is syndetic.

\begin{lemma}
\label{lem_nonempty_implies_syndetic}
    Let $(X,T)$ be a minimal system, $d \in \N$, $p \in \Z[x]^d$ be essentially distinct, and $U_1, \ldots, U_d \subseteq X$ be nonempty and open.  If the set $R_p(U_1, \ldots, U_d)$ is nonempty, then it is syndetic.
\end{lemma}

\begin{proof}
    Suppose $a \in R_p(U_1, \ldots, U_d)$.  Writing $p=(p_1,\ldots,p_d)$, define 
    \[V := T^{-p_1(a)} U_1 \cap \cdots \cap T^{-p_d(a)} U_d\;\;\text{and}\;\;q_i(n) := p_i(n+a) - p_i(a).\]  Note that $V$ is an open and nonempty subset of each $T^{-p_i(a)} U_i$. Combining this observation with \cref{lem_translated_poly_returns}, we see that $q = (q_1,\ldots,q_d) \in \Pol_d$ and that
    \[R_q(V,\ldots,V) \subseteq R_q(T^{-p_1(a)}U_1, \ldots, T^{-p_d(a)}U_d) = R_p(U_1,\ldots,U_d) - a.\]
    By \cite[Lemma 6.12]{Bergelson_McCutcheon00}, the set $R_q(V,\ldots,V)$ is syndetic, implying $R_p(U_1,\ldots,U_d)$ is syndetic as well.
\end{proof}

The following theorem of Qiu was discussed in the introduction; formulated in an equivalent way in \cref{thm_qiu_eqiuvalent_statement}, it is understood as a polynomial recurrence lifting theorem.  We formulate a version of it in \cref{thm_QiuC} that will be more convenient for our purposes, then we prove that it is indeed equivalent to \cref{thm_qiu_eqiuvalent_statement}.

\begin{theorem}[{\cite[Theorem B]{Qiu23}}]
\label{thm_QiuB}
Let $(X, T)$ be minimal and let $\pi: X \to X_{\infty}$ be the factor map to its $\infty$-step pronilfactor. There exist minimal systems $(X^*,T)$ and $(X_{\infty}^*,T)$ which are almost 1-1 extensions of $X$ and $X_{\infty}$ respectively, and a commutative diagram of factor maps as in \eqref{tikz_comm_diagram} such that for all open subsets $U_1, \ldots, U_d \subseteq X^*$ with $\pi^*(U_1)\cap\cdots\cap \pi^*(U_d)\neq \emptyset$ and all $p\in \Pol_d$,  we have $R_p(U_1,\ldots,U_d)\neq \emptyset$.
\end{theorem}

\begin{theorem}
\label{thm_QiuC}
Let $(X, T)$ be a minimal system, and let $\pi: X \to Y$ contain the $\infty$-step pronilfactor of $X$.  Suppose that $\pi$ is open.  For all nonempty, open $U_1, \ldots, U_d \subseteq X$ with $\pi(U_1)\cap\ldots\cap \pi(U_d)\neq \emptyset$ and all $p \in \Pol_d$, the set $R_p(U_1,\ldots,U_d)$ is syndetic.
\end{theorem}

\begin{proof}
Let $\rho: Y \to X_{\infty}$ be the factor map from $Y$ to the $\infty$-step nilfactor $X_{\infty}$ of $X$. In light of \cref{thm_QiuB}, we have the following commutative diagram:
\begin{center}
    \begin{tikzcd}[row sep = large, column sep = large]
        X \ar[r, leftarrow, "\tau"] \ar[d, rightarrow, "\pi"] &  X^* \ar[dd, rightarrow, "\pi^*"]\\
        Y \ar[d, rightarrow, "\rho"]& 
        \\
        X_{\infty} \ar[r, leftarrow, "\sigma"] & X_{\infty}^*
    \end{tikzcd}
\end{center}
Let $U_1, \ldots, U_d$ be nonempty, open subsets of $X$ with $\bigcap_{i=1}^d \pi(U_i)$ is a nonempty, open subset of $Y$. Since the map $\rho$ is semiopen, the interior of $\bigcap_{i=1}^d \rho(\pi(U_i))$ is nonempty. The map $\sigma: X_{\infty}^* \to X_{\infty}$ is almost 1-1 and so there exists $x \in \bigcap_{i=1}^d \rho(\pi(U_i))$ such that $\sigma^{-1}\{x\}$ is a singleton in $X_{\infty}^*$.

Now because the diagram above commutes,
\[
    x \in \bigcap_{i=1}^d \rho(\pi(U_i)) =\bigcap_{i=1}^d \sigma(\pi^*(\tau^{-1}(U_i))).
\]
It follows that 
\[
    \sigma^{-1}\{x\} \in \bigcap_{i=1}^d \pi^*(\tau^{-1}(U_i)).
\]
In particular, the right hand side is nonempty. Writing $p=(p_1,\ldots,p_d)$, by \cref{thm_QiuB} there exists $n \in \Z$ such that
\[
    T^{-p_1(n)} \tau^{-1}(U_1) \cap \cdots \cap T^{-p_d(n)} \tau^{-1}(U_d) \neq \emptyset.
\]
In particular, $R_p(\tau^{-1} U_1, \ldots, \tau^{-1} U_d) \neq \emptyset$. Since $T$ and $\tau$ commute, the set $R_p(U_1, \ldots, U_d)$ is nonempty.  It follows then by \cref{lem_nonempty_implies_syndetic} that the set $R_p(U_1, \ldots, U_d)$ is syndetic.
\end{proof}

We conclude this subsection with a proof that the topological characteristic factor result of Qiu in \cref{thm_QiuB} is equivalent to the one formulated in \cref{thm_qiu_eqiuvalent_statement} in the introduction.

\begin{theorem}
\label{thm_qiu_equivalent}
\cref{thm_qiu_eqiuvalent_statement} and 
\cref{thm_QiuB} are equivalent.
\end{theorem}

\begin{proof}
    (\cref{thm_qiu_eqiuvalent_statement} implies \cref{thm_QiuB}) \ By the O-diagram construction (see \cite[Theorem~3.1]{Veech-pointDistalFlow} or \cite[Theorem~2.19]{glasner_huang_shao_weiss_ye_2020}), there exist minimal systems $(X^*,T)$ and $(X_{\infty}^*,T)$ which are almost 1-1 extensions of $X$ and $X_{\infty}$ respectively, and a commutative diagram of factor maps as in \eqref{tikz_comm_diagram}.
    Let $U_1, \ldots, U_d \subseteq X^*$ be open and satisfy $\pi^*(U_1)\cap\ldots\cap \pi^*(U_d)\neq \emptyset$.  
    Let $p\in \Pol_d$.

    It is the case that $X^*_\infty$ is the $\infty$-step pronilfactor of $X^*$ and that $\pi^*: X^* \to X^*_\infty$ is the associated factor map (see \cite[Lemma 5.6]{glasner_huang_shao_weiss_ye_2020}).  Note that $\pi^*(U_1)\cap\ldots\cap \pi^*(U_d)\neq \emptyset$ implies that $0 \in R_p(\pi^*(U_1), \ldots, \pi^*(U_d))$; in particular, the set $R_p(\pi^*(U_1), \ldots, \pi^*(U_d))$ is nonempty. Since $\pi^*$ is open, $(\pi^*(U_i))^\circ = \pi^*(U_i)$, so it follows from \cref{thm_qiu_eqiuvalent_statement} that $R_p(U_1, \ldots, U_d)$ is nonempty, as was to be shown.\\

    (\cref{thm_QiuB} implies \cref{thm_qiu_eqiuvalent_statement}) \ Let $d \in \N$, $U_1, \ldots, U_d \subseteq X$ be nonempty and open, and $p \in \Z[x]^d$ be essentially distinct.  Suppose $R_p((\pi U_1)^\circ, \ldots, (\pi U_d)^\circ) \neq \emptyset$.  We must show that $R_p(U_1, \ldots, U_d)$ $\neq \emptyset$.

    By \cref{thm_QiuB}, there exist minimal systems $X^*$ and $X_{\infty}^*$ which are almost 1-1 extensions of $X$ and $X_{\infty}$ respectively, and a commuting diagram as in \eqref{tikz_comm_diagram}. Since $\tau$ is surjective, we see that
    \[\sigma \pi^* \tau^{-1} U_i = \pi \tau \tau^{-1} U_i = \pi U_i.\]
    By our assumptions, the previous line, and \cref{lemma_multiple_set_returns_lift_in_almost_11_extensions}, we see that
    \[\emptyset \neq R_p((\pi U_1)^\circ, \ldots, (\pi U_d)^\circ) = R_p((\sigma \pi^* \tau^{-1} U_1)^\circ, \ldots, (\sigma \pi^* \tau^{-1} U_d)^\circ) = R_p(\pi^* \tau^{-1} U_1, \ldots, \pi^* \tau^{-1} U_d).\]
    Let $n \in R_p(\pi^* \tau^{-1}U_1, \ldots, \pi^* \tau^{-1}U_d)$, and note that
    \[\pi^* T^{-p_1(n)} \tau^{-1}U_1 \cap \cdots \cap \pi^* T^{-p_d(n)} \tau^{-1}U_d \neq \emptyset.\]
    Define $q \in \Z[x]^d$ by $q_i(m) = p_i(m+n) - p_i(n)$.  
    By \cref{lem_translated_poly_returns}, $q\in \Pol_d$.
    By \cref{thm_QiuB} applied to the sets $T^{-p_1(n)}\tau^{-1} U_1, \ldots, T^{-p_d(n)}\tau^{-1} U_d \subseteq X^*$ and the polynomial tuple $q$, we have that
    \[R_q(T^{-p_1(n)}\tau^{-1} U_1, \ldots, T^{-p_d(n)}\tau^{-1} U_d) \neq \emptyset.\]
    By \cref{lem_translated_poly_returns}, we see that
    \[\emptyset \neq R_q(T^{-p_1(n)}\tau^{-1} U_1, \ldots, T^{-p_d(n)}\tau^{-1} U_d) = R_p(\tau^{-1} U_1, \ldots, \tau^{-1} U_d) - n = R_p(U_1, \ldots, U_d) - n.\]
    It follows that $R_p(U_1, \ldots, U_d) \neq \emptyset$, as was to be shown.
\end{proof}

\subsection{Proof of \texorpdfstring{\cref{mainthm_ps_structure_theory_pol}}{Theorem A}}

\label{sec:proof_a_open}

We conclude the proof of \cref{mainthm_ps_structure_theory_pol} by proving \cref{thm_ps_structure_theory_open}.  (This reduction was explained in \cref{sec_reduction_to_open}.)

\begin{proof}[Proof of \cref{thm_ps_structure_theory_open}]
Define
\[\family \defeq \ \uparrow \big\{ R_p(U_1,\ldots,U_d) : \ d\in\N, \ p\in\Pol_d, \ \text{$U_1,\ldots, U_d \subseteq X$ open, $\pi U_1\cap \cdots \cap \pi U_d \neq \emptyset$} \big\}.\]
To show that for all nonempty, open $U_1, \ldots, U_d \subseteq X$ and all $p\in\Pol_d$, the set
\[R_p(\pi U_1, \ldots, \pi U_d) \setminus R_p(U_1, \ldots, U_d)\]
is not piecewise syndetic, it suffices, via \cref{lemma_diff_not_ps}, to show:
\begin{enumerate}
    \item \label{item:sub-collection-F} the intersection of all sets in any finite sub-collection of $\family$ is syndetic; and 

    \item \label{item:subcollection-F-2} $R_p(U_1, \ldots, U_d)\subseteq R_p(\pi U_1, \ldots, \pi U_d) \subseteq R_p(U_1, \ldots, U_d)-\family.$
\end{enumerate}

Property \eqref{item:sub-collection-F} follows by combining \cref{lemma_special_sets_of_returns_form_a_filterpol} and \cref{thm_QiuC}.  To see why property \eqref{item:subcollection-F-2} holds, note first that by \cref{lemma_containment_in_one_direction},
\[R_p(U_1, \ldots, U_d)\subseteq R_p(\pi U_1, \ldots, \pi U_d).\]
Taking $a\in R_p(\pi U_1, \ldots, \pi U_d),$ where $p=(p_1,\ldots,p_d),$ we have 
\[
    \pi(T^{-p_1(a)}U_1)\cap \ldots\cap \pi(T^{-p_d(a)}U_d)\neq \emptyset.
\]
Defining $q=(q_1,\ldots,q_d)$ by $q_i(n) = p_i(n+a)-p_i(a),$ we have by \cref{lem_translated_poly_returns} that $q \in \Pol_d$ and
\[
    R_p(U_1, \ldots, U_d) - a = R_q(T^{-p_1(a)}U_1,\ldots, T^{-p_d(a)}U_d) \in \family.
\]
This proves that $R_p(\pi U_1, \ldots, \pi U_d) \subseteq R_p(U_1, \ldots, U_d)-\family$, as desired.
\end{proof}

\section{Applications}
\label{sec_applications}

In this section, we prove Theorems \ref{Mainthm:odd-recurrence-pol}, \ref{cor:pol_k_p}, \ref{cor:totally_minimal}, and \ref{mainthm_conjecture_are_equiv}.

\subsection{Criterion for disjointness}
\label{sec_disjointness_criterion}

In this subsection we establish an analogue for minimal systems of \cite[Theorem 3.1]{Berg71} in ergodic theory, describing a criterion for disjointness involving a distal system.

A \emph{joining} of two systems $(X, T)$ and $(Y, S)$ is a subsystem $(W, T \times S)$ of $(X \times Y, T \times S)$ such that $\pi_X(W) = X$ and $\pi_Y(W) = Y$, where $\pi_X, \pi_Y$ are the projections from $X \times Y$ to $X$ and $Y$, respectively. 
We say two systems are \emph{disjoint} if their only joining is the entire product system $(X \times Y, T \times S)$.
Points $x$ and $y$ in a system $(X,T)$ are \emph{proximal} if there exists $z \in X$ such that $(z,z) \in \orb_{T \times T}(x,y)$.  
A system $(X,T)$ is \emph{distal} if for all distinct $x, y \in X$, the points $x$ and $y$ are not proximal.

For the proof of the following proposition, recall the discussion from \cref{sec_prelims_top_dyn} regarding the maximal equicontinuous factor and eigenfunctions.

\begin{proposition}\label{lem:disjoint-mef}
Let $(X, T)$ be a minimal distal system and let $(Y, S)$ be a minimal system. The two systems are disjoint if and only if they do not share any nontrivial eigenvalues.    
\end{proposition}

\begin{proof}
    Combining \cite[Chapter 11, Theorem 7]{Auslander88} with \cite[Chapter 11, Proposition 2 (ii)]{Auslander88} it follows that $(X, T)$ and $(Y, S)$ are disjoint if and only if their maximal equicontinuous factors are disjoint.      By \cite[Chapter 11, Theorem 9]{Auslander88}, two equicontinuous systems are disjoint if and only if their product is transitive.  By \cite[Theorem 5]{Peleg72} (see also \cite{MR467704}), the product of two minimal systems is transitive if and only if they share no nontrivial eigenvalue.
    This proves that $(X,T)$ and $(Y,S)$ are disjoint if and only if their maximal equicontinuous factors share no eigenvalues in common.  The conclusion of the proposition now follows from the fact that the set of eigenvalues of a system coincides with the set of eigenvalues of its maximal equicontinuous factor, as was explained in \cref{sec_prelims_top_dyn}.
\end{proof}

\subsection{Eigenvalues of derived systems}

By the \emph{spectrum} of a system $(X,T)$ we mean its eigenvalues as a subgroup of $\R/\Z$.
The spectrum of $(X,T)$ is denoted $\sigma(X,T)$. 
Since $X$ is metrizable, the spectrum $\sigma(X, T)$ is countable.

\begin{lemma}
\label{lem:MEF_nilsystem_not_connected}
Let $(X, T)$ be a nilsystem with $k$-many connected components $X_0, X_1, \ldots, X_{k-1}$. The systems $(X_i, T^k)$ are pairwise isomorphic, and $\sigma(X, T) = \sigma(X_i, T^k)/k := \{\theta \in \R/\Z: k \theta \in \sigma(X_i, T^k)\}$.
\end{lemma}

\begin{proof}
For the first claim, observe that for $i, j$, $(X_i, T^k)$ is isomorphic to $(X_j, T^k)$ through the isomorphism $T^{j-i}$.

Suppose $T X_i = X_{i+1 \bmod k}$ and note that $(X, T)$ is isomorphic to $(X_0 \times \Z/k\Z, S)$ where $S: X_0 \times \Z/k\Z \to X_0 \times \Z/k\Z$ is defined by
\[
    S(x, n) = \begin{cases}
        (x, n + 1) & \text{ if } 0 \leq n \leq k - 2, \\
        (T^k x, 0) & \text{ if } n = k - 1.
    \end{cases}
\]
If $f:X_0\times\Z/k\Z\to\C$ is a continuous eigenfunction with $f \circ S = e^{2\pi i\theta} f$, then the restriction $f'=f|_{X_0}$ to $X_0$ is an eigenfunction for $(X_0,T^k)$ with eigenvalue $k\theta$, showing that $\sigma(X,T) \subseteq \sigma(X_0,T^k)/k$.

Conversely, if $\theta\in\sigma(X_0,T^k)/k$ and $f:X_0\to\C$ is a continuous eigenfunction of $(X_0,T^k)$ with eigenvalue $k\theta$, the function $f':X_0\times(\Z/k\Z)\to\C$ given by $f'(x,n)=e^{2\pi i n\theta}f(x)$ satisfies $f' \circ S = e^{2\pi i \theta}f'$, showing that $\sigma(X_0,T^k)/k \subseteq \sigma(X,T)$.
\end{proof}

\begin{lemma}\label{lem:same-mef}
    Let $(X, T)$ and $(Y,S)$ be two minimal systems that do not share any nontrivial eigenvalues, and suppose $(X,T)$ is a nilsystem.
    Let $d \in \N$ and $\tilde{T} = I \times T \times T^2 \times \cdots \times T^d$.
For $x \in X$, define
\begin{align}
\label{eqn_def_of_zx}
    Z_x : = \overline{o}_{\tilde{T}}(x,\ldots,x) = \overline{\{\tilde{T}^n(x, \ldots, x): n \in \Z\}} \subseteq X^{d+1}.
\end{align}
For almost every $x \in X$ with respect to the Haar measure on $X$, the system $(Z_x, \tilde{T})$ does not share any nontrivial eigenvalues with $(Y,S)$.
\end{lemma}

\begin{proof}
First, assume $X$ is connected. Revised Theorem~7.1 in \cite{Moreira_Richter-spectral} states that for any $\theta$ which is not an eigenvalue of $(X, T)$, for almost every $x \in X$, $\theta$ is not an eigenvalue for $(Z_x, \tilde{T})$. For each $\theta \in \R/\Z$, let 
\[
    A_{\theta} = \{x \in X: \theta \not \in \sigma(Z_x, \tilde{T})\};
\]
and define
\[
    A = \bigcap_{\theta \in \sigma(Y, S) \setminus \sigma(X, T)} A_{\theta} = \bigcap_{\theta \in \sigma(Y, S) \setminus \{0\}} A_{\theta}.
\]
Since $\sigma(Y, S)$ is at most countable, $A \subseteq X$ has full Haar measure. Our lemma now follows because for each $x \in A$, $\sigma(Z_x, \tilde{T}) \cap \sigma(Y, S) = \{0\}$.

Now suppose that $(X, T)$ is a minimal nilsystem with $k$ connected components $X_0, X_1, \ldots, X_{k-1}$ satisfying $T X_i = X_{i + 1 \bmod k}$. For each $i$, the system $(X_i, T^k)$ is a connected minimal nilsystem and so $\sigma(X_i, T^k)$ does not contain a nonzero rational eigenvalue. Moreover, by \cref{lem:MEF_nilsystem_not_connected}, $\sigma(X, T) = \sigma(X_i, T^k)/k$ and so $(X_i, T^k)$ does not share any nontrivial eigenvalue with $(Y, S)$.

By the proof for the case of connected $(X, T)$, there exists a set $B_0 \subseteq X_0$ of full measure (with respect to $X_0$) such that for all $x \in B_0$, the system $(Z_{x,0} := \overline{\{\tilde{T}^{kn}(x, \ldots, x): n \in \Z\}}, \tilde{T}^k)$ does not have any nonzero rational eigenvalue (and so $Z_{x,0}$ is connected) and does not share any nontrivial eigenvalue with $(Y, S)$. 

For all $x \in X_0$, we have
\[
    Z_x := \overline{\{\tilde{T}^n(x, \ldots, x): n \in \Z\}} = Z_{x, 0} \cup \tilde{T} Z_{x, 0} \cup \cdots \cup \tilde{T}^{k-1} Z_{x, 0}.
\]
Therefore, for all $x \in B_0$, $Z_x$ has $k$ connected components $Z_{x, 0}, \tilde{T} Z_{x, 0}, \ldots, \tilde{T}^{k-1} Z_{x, 0}$. By \cref{lem:MEF_nilsystem_not_connected} 
\begin{equation}\label{eq:divide_sigma}
    \sigma(Z_x, \tilde{T}) = \sigma(Z_{x,0}, \tilde{T}^k)/k.
\end{equation}

Let $\gamma$ be an arbitrary element in $\sigma(Z_x, \tilde{T}) \cap \sigma(Y, S)$. By \eqref{eq:divide_sigma}, $k \gamma = \beta$ for some $\beta \in \sigma(Z_{x,0}, \tilde{T}^k)$. Since $\sigma(Y, S)$ is a group, $\beta = k \gamma \in \sigma(Y, S)$.
Because
\[
    \sigma(Z_{x,0}, \tilde{T}^k) \cap \sigma(Y, S) = \{0\},
\]
we have $\beta = 0$. In other words, $k \gamma = 0 \in \sigma(X_i, T^k)$ (as $0$ is an eigenvalue for every system), and so $\gamma \in \sigma(X, T)$. Since $(X, T)$ does not share any nontrivial eigenvalue with $(Y, S)$, $\gamma = 0$. It follows that for $x \in B_0$,
\begin{equation}\label{eq:trivial_eigenspaces}
    \sigma(Z_x, \tilde{T}) \cap \sigma(Y, S) = \{0\}.
\end{equation}

Let $B = B_0 \cup T B_0 \cup \cdots \cup T^{k-1} B_0$. Then $B$ is a set of full measure in $X$. Suppose $x = T^i x_0$ be an arbitrary element in $B$ with $x_0 \in B_0$. Then $(Z_x, \tilde{T}^k)$ is isomorphic to $(Z_{x_0}, \tilde{T}^k)$ through the map $\tilde{T}^{-i}$. As a result, \eqref{eq:trivial_eigenspaces} holds for every $x \in B$ and so our lemma is proved.
\end{proof}

\subsection{Proof of \texorpdfstring{\cref{Mainthm:odd-recurrence-pol}}{Theorem B}}
\label{sec_proof_of_thm_b}

\begin{lemma}
\label{lem:odd-rec-pol-nilsystems}
Let $(X, T)$ and $(Y, S)$ be minimal systems. Assume further that $(X, T)$ is a nilsystem and the two systems do not share any nontrivial eigenvalues.
For all nonempty, open sets $U \subseteq X, V \subseteq Y$, and every $y \in Y$ (not necessarily in $V$), the set
\[\big\{ n \in \Z : \ U \cap T^{-n} U \cap \cdots \cap T^{-dn} U \neq \emptyset \big\} \cap R_S(y, V)\]
is dynamically syndetic.
\end{lemma}

\begin{proof}
Because the Haar measure on $X$ has full support, by \cref{lem:same-mef}, there exists $x \in U$ such that the minimal nilsystem $(Z_x, \tilde T)$ described in \eqref{eqn_def_of_zx} and $(Y, S)$ do not share any nontrivial eigenvalues. Since nilsystems are distal, by \cref{lem:disjoint-mef}, the systems $(Z_x, \tilde T)$ and $(Y, S)$ are disjoint, and so their product $(Z_x \times Y, \tilde{T} \times S)$ is minimal.

Let $x_1 = (x, \ldots, x)$ and $W = U \times \cdots \times U$.
Since $W$ is a neighborhood of $x_1$ in $Z_x$, the set $(W \times V) \cap (Z_x \times Y)$ is a nonempty, open subset of $Z_x \times Y$. We then have
\[
     \big\{ n \in \Z : \ U \cap T^{-n} U \cap \cdots \cap T^{-dn} U \neq \emptyset \big\} \cap R_S(y, V) \supseteq R_{\tilde{T} \times S}\big((x_1, y), W \times V \big)
\]
which is a dynamically syndetic set, as desired.
\end{proof}

We are now ready to prove \cref{Mainthm:odd-recurrence-pol}, restated here for convenience.

\begin{named}{\cref{Mainthm:odd-recurrence-pol}}{}
Let $(X, T)$ and $(Y, S)$ be minimal systems that do not share any nontrivial eigenvalues.
For all nonempty, open sets $U \subseteq X$ and $V \subseteq Y$, all $y \in Y$ (not necessarily in $V$), and all $d \in \N$, the set
\[\big\{ n \in \Z: \ U \cap T^{-n} U \cap \cdots \cap T^{-dn} U \neq \emptyset \big\} \cap R_S(y,V)\]
is dynamically syndetic.
\end{named}

\begin{proof}[Proof of \cref{Mainthm:odd-recurrence-pol}]

Let $(X_{\infty}, T)$ be the $\infty$-step pronilfactor of $(X, T)$ and let $\pi: (X, T) \to (X_{\infty}, T)$ be the factor map. Let $U$ be a nonempty, open subset of $X$ and $q \in \Z[x]^{d+1}$ be the polynomial tuple $(0,x, 2x, \ldots, dx)$. \cref{mainthm_ps_structure_theory_pol} states that
\[
    R_q((\pi U)^{\circ}, \ldots, (\pi U)^{\circ}) \ \setminus \ R_q(U, \ldots, U) 
\]
is not piecewise syndetic. Therefore, the set
\[
    \big( R_q((\pi U)^{\circ}, \ldots, (\pi U)^{\circ}) \cap R_S(y, V) \big) \ \setminus  \ \big(R_q(U, \ldots, U) \cap R_S(y, V) \big)
\]
is not piecewise syndetic.
In view of \cref{lemma_pwssetdifference}, it remains to show that $R_q((\pi U)^{\circ}, \ldots, (\pi U)^{\circ}) \cap R_S(y, V)$ is dynamically syndetic.

By \cref{lemma_semiopen}, the map $\pi$ is semi-open, so the set $\pi U$ has nonempty interior.
As $X_\infty$ is an inverse limit of nilsystems, we can find (cf.\ \cref{rmk_open_set_in_an_inverse_limit}) a nilsystem factor $\phi : (X_\infty, T) \to (Z, T)$ and an open set $V \subseteq Y$ such that $\phi^{-1}(V) \subseteq (\pi U)^\circ$.
As a result,
\[
    R_q(V, \ldots, V) \subseteq R_q((\pi U)^{\circ}, \ldots, (\pi U)^{\circ}).
\]
Moreover, since $(Z, T)$ is a factor of $(X, T)$, the systems $(Z, T)$ and $(Y,S)$ do not share any nontrivial eigenvalues.
Now applying \cref{lem:odd-rec-pol-nilsystems} to the nilsystem $(Z, T)$ with the open set $V$, we arrive at the desired result.
\end{proof}

\subsection{Proof of \texorpdfstring{\cref{cor:pol_k_p}}{Theorem C}}
\label{sec_proof_of_theorem_b}

To prove \cref{cor:pol_k_p}, we first establish the result in nilsystems (\cref{lem:odd-rec-pol-nilsystems-ap}), and then we apply \cref{mainthm_ps_structure_theory_pol} to ``lift'' the result to arbitrary systems.
We begin with a few lemmas about polynomial orbits in nilsystems.

\begin{lemma}\label{lem:same-mef3}
    Let $(X, T)$ be a minimal nilsystem and for $d \in \N$, let $\tilde{T} = I \times T \times T^2 \times \cdots \times T^d$.
For $x \in X$, define
\[
    Z_x : = \overline{o}_{\tilde{T}}(x,\ldots,x) = \overline{\big\{\tilde{T}^n(x, \ldots, x): \ n \in \Z\big\}} \subseteq X^{d+1}.
\]
If $X$ is connected, then for almost every $x\in X$, the nilmanifold $Z_x$ is connected.
\end{lemma}

\begin{proof}
    Since $X$ is connected, the nilsystem $(X,T)$ has no nonzero rational eigenvalues. 
    \cite[Revised Theorem 7.1]{Moreira_Richter24} implies that for almost every $x$ the nilsystem $(Z_x,\tilde T)$ has no nonzero rational eigenvalue.
    This implies, for any such $x$, that $(Z_x,\tilde T)$ is totally minimal, and hence that $Z_x$ is connected.
\end{proof}

\begin{lemma}[{\cite[Proposition 2.7]{Frantzikinakis08}}]
\label{lem:frantzikinakis_connected}
Suppose that $X = G/\Gamma$ is a nilmanifold, $g: \N \to G$ is a polynomial sequence and $x \in X$ is such that $Y = \overline{\{g(n) x: \ n \in \Z\}}$ is connected. Then for any non-constant polynomial $p$ on $\Z$ we have $Y = \overline{\{g(p(n)) x: \ n \in \Z\}}$.
\end{lemma}

\begin{lemma}\label{lem:polynomial_in_nil_dynamical_syndetic}
Let $(X, T)$ be a nilsystem and $p$ be a non-constant polynomial on $\Z$. For $x \in X$, and a nonempty open set $U \subseteq X$, if $\{n \in \Z: T^{p(n)} x \in U\}$ is nonempty, it is a dynamically syndetic set.
\end{lemma}
\begin{remark}
As an illustration for this lemma, let $(X, T)$ be the rotation by $\alpha$ (rational or irrational) on $1$-dimension torus $\T$ and $p(n) = n^2$. 
Consider the new system $(Y, S)$ with $Y = \T \times \T$ and $S(x, y) = (x + \alpha, y + 2x + \alpha)$. Then
\[
    S^n(0, 0) = (n \alpha, n^2 \alpha)
\]
and so for any open set $U \subseteq \T$,
\[
     \{n \in \Z: T^n(0) \in U\} = \{n \in \Z: S^n(0, 0) \in \T \times U\}
\]
which is a dynamically syndetic set if it is nonempty.
\end{remark}
\begin{proof}
Suppose $X = G/\Gamma$ where $G$ a nilpotent Lie group and $\Gamma$ is a closed, discrete, cocompact subgroup of $G$. Write $T(x) = ax$  for some $a \in G$. Let $\pi: G \to X$ be the projection map and let $x$ be as in the lemma's statement.

Without loss of generality, we will assume $x = \pi(1_G)$ where $1_G$ is the identity element of $G$. Indeed, suppose $x = h \Gamma$. Define $\Gamma_x = h \Gamma h^{-1}$. Then $\phi: X \to G/\Gamma_x$ defined by $\phi(y) = hy$ is a homeomorphism between $X$ and $G/\Gamma_x$ and under this map $\pi(1_G) \mapsto h \pi(1_G) = x$.

By \cite[Proposition 3.14]{Leibman05a}, there exist a nilpotent Lie group $\tilde{G}$ with a discrete closed cocompact subgroup $\tilde{\Gamma}$, a surjection $\eta: \tilde{G} \to G$ with $\eta(\tilde{\Gamma}) \subseteq \Gamma$, a unipotent automorphism $\tau$ of $\tilde{G}$ with $\tau(\tilde{\Gamma}) = \tilde{\Gamma}$, and an element $c \in \tilde{G}$ such that $a^{p(n)} = \eta(\tau^n(c))$ for all $n \in \Z$.

Let $\tilde{X} = \tilde{G}/\tilde{\Gamma}$. The map $\eta: \tilde{G} \to G$ factors to $\eta: \tilde{X} \to X$, so that if $\tilde{\pi}: \tilde{G} \to \tilde{X}$ is the projection map, then $\pi \circ \eta = \eta \circ \tilde{\pi}$. Letting $\tilde{x} = \tilde{\pi}(1_{\tilde{G}})$, we have $\eta(\tau^n (c \tilde{x})) = a^{p(n)} x$ for $n \in \Z$. As a result,
\[
    \{n \in \Z: a^{p(n)} x \in U\} = \{n \in \Z: \tau^n (c \tilde{x}) \in \eta^{-1}(U)\}.
\]
Note that the left hand side is nonempty according to our lemma's assumption.

Let $\hat{G}$ be the extension of $\tilde{G}$ by $\tau$. By \cite[Proposition 3.9]{Leibman05a}, $\hat{G}$ is a nilpotent Lie group. Let $\hat{\tau}$ be the element in $\hat{G}$ representing $\tau$ in the sense that $\tau(h) = \hat{\tau} h \hat{\tau}^{-1}$ for any $h \in \tilde{G}$. Let $\hat{\Gamma} = \langle \tilde{\Gamma}, \hat{\tau} \rangle \subseteq \hat{G}$. Since $\tau(\tilde{\Gamma}) = \tilde{\Gamma}$, we have $\hat{\Gamma} \cap \tilde{G} = \tilde{\Gamma}$, and so $\hat{\Gamma}$ is a discrete subgroup of $\hat{G}$ and $\tilde{X} = \hat{G}/\hat{\Gamma}$. Therefore, as a transformation on $\tilde{X}$, $\hat{\tau}$ is the same as $\tau$, i.e., for any $h \in \tilde{G}$ and $y = h \hat{\Gamma} \in \tilde{X}$, we have 
\[
    \tau(y) = \tau(h) \hat{\Gamma} = \hat{\tau} h \hat{\tau}^{-1} \hat{\Gamma} = \hat{\tau} h \hat{\Gamma} = \hat{\tau} y.
\]
Thus,
\[
    \{n \in \Z: \tau^n (c \tilde{x}) \in \eta^{-1}(U)\} = \{n \in \Z: \hat{\tau}^n (c \tilde{x}) \in \eta^{-1}(U)\}.
\]
Since $(\tilde{X}= \hat{G}/\hat{\Gamma}, \hat{\tau})$ is a nilsystem and the set in the right hand side is nonempty, this set is a dynamically syndetic set, as was to be shown.
\end{proof}

The next lemma is a special case of Theorems \ref{cor:pol_k_p} and \ref{cor:totally_minimal} and an important ingredient in the proofs of both.

\begin{lemma}\label{lem:totall_minimal_nilsystems}
    Let $(X,T)$ be a totally minimal nilsystem. For all $d \in \N$, all open $\emptyset\neq U\subseteq X$, and all non-constant polynomial $p$ on $\Z$, the set $R_{(0,p,2p,\ldots,dp)} (U, \ldots, U)$ is dynamically syndetic.
\end{lemma}

\begin{proof}
Denote 
\[
    \tilde{T} = I \times T \times \cdots \times T^d.
\]
Given a point $x_0 \in X$, let  
\[
    x_1 = (x_0, \ldots, x_0)
\]
 and let $\tilde{Z}$ be the orbit closure of $x_1$ under the map $\tilde{T}$. Then $(\tilde{Z}, \tilde{T})$ is a minimal nilsystem. 
By \cref{lem:same-mef3}, we can choose $x_0 \in U$ such that $\tilde{Z}$ is connected. Note that for this $x_0$, we also have $\tilde{Z} \cap U^{d+1}$ is a nonempty open subset of $\tilde{Z}$. 

By \cref{lem:frantzikinakis_connected}, the sequence $\big(\tilde{T}^{p(n)} x_1\big)_{n\in\Z}$ is dense in $\tilde{Z}$, and so $\{n \in \Z: \tilde{T}^{p(n)}(x_1) \in U^{d+1}\}$ is nonempty. By \cref{lem:polynomial_in_nil_dynamical_syndetic}, the set $\{n \in \Z: \tilde{T}(x_1) \in U^{d+1}\}$ is dynamically syndetic. For each
$n$ in this set, the intersection
$$U\cap T^{-p(n)}U\cap\cdots\cap T^{-dp(n)}U$$
contains $x_0$ and is thus nonempty.
\end{proof}

\begin{lemma}
\label{lem:odd-rec-pol-nilsystems-ap}
Let $(X, T)$ be a minimal nilsystem and $k, d \in \N$.  If the system $(X,T^k)$ is minimal, then for all open $\emptyset\neq U \subseteq Z$, all polynomials $p$ on $\Z$ with $p(0) = 0$, and all $j \in \Z$, the set
\[
    R_{(0,p,2p,\ldots,dp)} (U,\ldots,U) \cap (k \Z + j)
\]
is dynamically syndetic.   
\end{lemma}
\begin{remark}
It is worth mentioning the similarities between \cref{lem:odd-rec-pol-nilsystems-ap} and \cref{lem:odd-rec-pol-nilsystems}. Nevertheless, none of the lemmas imply the other: \cref{lem:odd-rec-pol-nilsystems-ap} applies to an arbitrary polynomial $p(n)$ (instead of linear polynomial $p(n) = n$ as in \cref{lem:odd-rec-pol-nilsystems}) but only to rotation on $k$ points (instead of arbitrary system $(Y, S)$).
\end{remark}

\begin{proof}
    Let $\ell$ be the number of connected components of $X$.
    The assumption that $(X,T^k)$ is minimal is equivalent to the co-primality relation $(\ell,k)=1$.
    It follows that there exists $s\in\N$ such that $sk+j$ is a multiple of $\ell$.  Since $p(0) = 0$, the polynomial $q(n):=\tfrac1\ell p\big(k(\ell n+s)+j\big)$ takes values in $\Z$.

    Let $X_0\subseteq X$ be a connected component with $U_0:=X_0\cap U\neq\emptyset$.
    Then $(X_0,T^\ell)$ is a totally minimal nilsystem.
    It follows by \cref{lem:totall_minimal_nilsystems} that
    the set
    \[
        A \defeq \big\{ n \in \Z : U_0 \cap (T^\ell)^{-q(n)} U_0 \cap \cdots \cap (T^\ell)^{-dq(n)} U_0 \neq \emptyset \big\}
    \]
    is dynamically syndetic. By \cite[Lemmas 3.3 and 3.4]{glasscock_le_2024}, the set $k \big( \ell A + s \big) + j$ is also dynamically syndetic. Since
    \[
        k \big( \ell A + s \big) + j \subseteq R_{(0,p,2p,\ldots,dp)} (U,\ldots,U) \cap (k \Z + j),\]
    the conclusion of our lemma follows.
\end{proof}

We are now ready to prove \cref{cor:pol_k_p}, restated here for convenience.

\begin{named}{\cref{cor:pol_k_p}}{}
Let $(X,T)$ be a system and $k, d \in \N$. If the system $(X,T^k)$ is minimal, then for all $p \in \Z[x]$ with $p(0) = 0$, all nonempty, open $U \subseteq X$, and all $j \in \Z$, the set
\[R_{(0,p,2p,\ldots,dp)}(U,\ldots,U) \cap (k \Z + j)\]
is dynamically syndetic.
\end{named}

\begin{proof}
Assume the system $(X,T^k)$ is minimal, and let $p \in \Z[x]$ with $p(0) = 0$, $U \subseteq X$ be nonempty and open, and $j \in \Z$.  Denote by $q$ the polynomial tuple $(0, p, 2p, \ldots, dp) \in \Pol_{d+1}$.
Let $(X_\infty, T)$ be the $\infty$-step pronilfactor of $(X, T)$, and let $\pi: (X, T) \to (X_\infty, T)$ be the factor map.

By \cref{lemma_semiopen}, the map $\pi$ is semi-open, so the set $\pi U$ has nonempty interior.
As $X_\infty$ is an inverse limit of nilsystems, we can find (cf.\ \cref{rmk_open_set_in_an_inverse_limit}) a nilsystem factor $\phi : (X_\infty, T) \to (Y,T)$ and an open set $V \subseteq Y$ such that $\phi^{-1}(V) \subseteq (\pi U)^\circ$.

It follows from \cref{lem:odd-rec-pol-nilsystems-ap} that the set $R_q(V,\dots,V) \cap (k\Z + j)$ is dynamically syndetic.
The inclusion $\phi^{-1}(V) \subseteq (\pi U)^\circ$ gives $R_q(V, \ldots, V) \subseteq R_q((\pi U)^\circ, \ldots, (\pi U)^\circ)$, so
\begin{equation}
\label{eqn:something_is_syndetic}
    R_q((\pi U)^\circ, \ldots, (\pi U)^\circ) \cap (k\Z + j) \textup{ is dynamically syndetic.}
\end{equation}
Now \cref{mainthm_ps_structure_theory_pol} states that $R_q((\pi U)^\circ, \ldots, (\pi U)^\circ) \setminus R_q(U, \ldots, U)$ is not piecewise syndetic, so
\begin{equation}
\label{eqn:something_is_not_piecewise_syndetic}
R_q((\pi U)^\circ, \ldots, (\pi U)^\circ) \cap (k\Z + j)  \setminus R_q(U, \ldots, U) \cap (k\Z + j) \textup{ is not piecewise syndetic.}
\end{equation} 
Applying \cref{lemma_pwssetdifference} to \eqref{eqn:something_is_syndetic} and \eqref{eqn:something_is_not_piecewise_syndetic} concludes the proof.
\end{proof}

\subsection{Proof of \texorpdfstring{\cref{cor:totally_minimal}}{Theorem D}}
\label{sec_proof_of_thm_c}

We now prove \cref{cor:totally_minimal} which is restated here for convenience.

\begin{named}{\cref{cor:totally_minimal}}{}
    Let $(X,T)$ be a totally minimal system and $d \in \N$. For all non-constant $p \in \Z[x]$ and all nonempty, open $U\subseteq X$, the set $R_{(0,p,2p,\ldots,dp)}(U,\ldots,U)$ is dynamically syndetic.
\end{named}

\begin{proof}
Let $p \in \Z[x]$ be non-constant and $U \subseteq X$ be nonempty and open.  Let $q$ be the polynomial tuple $(0, p, 2p, \ldots, dp)$.
Let $(X_{\infty}, T)$ be the $\infty$-step pronilfactor of $(X, T)$, and let $\pi: (X, T) \to (X_{\infty}, T)$ be the factor map. 

\cref{mainthm_ps_structure_theory_pol} states that the set
\[
    R_q\big((\pi U)^\circ, \ldots, (\pi U)^\circ\big) \setminus R_q(U, \ldots, U) 
\]
is not piecewise syndetic. 
In view of \cref{lemma_pwssetdifference}, to show that $R_q(U, \ldots, U)$ is dynamically syndetic, it suffices to show that $R_q((\pi U)^\circ, \ldots, (\pi U)^\circ)$ is dynamically syndetic.

By \cref{lemma_semiopen}, the map $\pi$ is semi-open, so the set $\pi U$ has nonempty interior.
As $X_\infty$ is an inverse limit of nilsystems, we can find (cf.\ \cref{rmk_open_set_in_an_inverse_limit}) a nilsystem factor $\phi : (X_\infty, T) \to (Y,T)$ and an open set $V \subseteq Y$ such that $\phi^{-1}(V) \subseteq (\pi U)^\circ$.
As a result,
\[
    R_q(V, \ldots, V) \subseteq R_q\big((\pi U)^\circ, \ldots, (\pi U)^\circ\big).
\]
The system $(Y,T)$ is a totally minimal nilsystem, so \cref{lem:totall_minimal_nilsystems} gives that the set $R_q(V, \ldots, V)$ is dynamically syndetic, concluding the proof.
\end{proof}

\subsection{Relationship between conjectures in \texorpdfstring{\cite{glasner_huang_shao_weiss_ye_2020}}{[5]} and \texorpdfstring{\cite{Leibman05b}}{[11]}}

In this section we prove that Conjectures \ref{conj:polynomial-odd-recurrence} and \ref{conj:leibmanish} are equivalent. We do so by relating both to new conjectures on topological multiple recurrence and nilsystems.
Throughout the section, $(X,T)$ denotes a fixed minimal nilsystem.

\subsubsection{Technical results}

To prove the equivalence, we need some technical results about nilrotations from \cite{Leibman07}. 
Let $(X = G/\Gamma, T)$ be a nilsystem. A \emph{subnilmanifold} $Y$ of $X$ is a closed subset of $X$ for the form $H x$ where $H$ is a closed subgroup of $G$ and $x \in X$.

If $V$ is a connected subnilmanifold of $X$ and $g$ is a polynomial sequence in $G$, it is shown in \cite[Theorem 2.3]{Leibman07} that there exists a subset $Y_{V, g}$ of $X$ such that for almost every $x \in V$ (with respect to the Haar measure on $V$), the orbit closure $\orb_g(x) = \overline{\{g(n) x: n \in \Z\}}$ is a translate of $Y_{V, g}$, i.e. $\orb_g(x) = a Y_{V, g}$ for some $a \in G$.
A point $x$ with this property is called a \emph{generic point} for $g$ on $V$ and the set $Y_{V, g}$ is called the \emph{generic orbit} of $g$ on $V$.

\begin{lemma}[{\cite[Theorem 5.9]{Leibman07}}]
\label{lem:leibman_connected_components_of_translate_subnilmanifold}
Suppose $X = G/\Gamma$ is a nilmanifold and $g$ is a polynomial sequence in $G$. Let $V$ be a connected subnilmanifold of $X$ and let $Y_{V, g}$ be the generic orbit of $g$ on $V$. Let $Y$ be a connected component of $Y_{V, g}$ and suppose $Y = Hx,$ where $H$ is a closed subgroup of $G$ and $x \in X$. 
Then every connected component of the orbit closure $\orb_g(V) = \overline{\{g(n) V: n \in \Z\}}$ is a translate of $HV = \bigcup_{h \in H} hV$.
\end{lemma}

\begin{proposition}\label{prop:orbit_closure_subnil_arith_same}
Let $X = G/\Gamma$ be a nilmanifold and $g$ be a polynomial sequence in $G$. Let $V$ be a connected subnilmanifold of $X$. If $\orb_g(V)$ is connected, then
\[
    \orb_g(V) = \overline{\{g(p(n)) V: \ n \in \Z\}}
\]
for every non-constant polynomial $p$ on $\Z$.
\end{proposition}
\begin{proof}
Suppose $\orb_g(V)$ is connected and let $p$ be a non-constant polynomial over $\Z$.
Let $x \in V$ be generic for both polynomial sequences $g(n)$ and $h(n) = g(p(n))$. (Note that such a point $x$ exists since the set of generic points on $V$ for each polynomial sequence has full Haar measure.)

By \cite[Theorem B]{Leibman05a}, there exists a connected closed subgroup $H$ of $G$ and points $x_0, \ldots, x_{\ell-1} \in X$, not necessarily distinct, such that the sets $Y_i = H x_i$ are closed subnilmanifolds of $X$, and
\[
    \orb_g(x) = \bigcup_{i=0}^{\ell-1} Y_i.
\]
Moreover, the sequence $n \mapsto g(n) x$ cyclically visits $Y_0, \ldots, Y_{\ell-1}$ and the sequence $n \mapsto g(\ell n + i)x$ is well-distributed on $Y_i$. 

For $i \in \{0, \ldots, \ell - 1\}$, define a new polynomial sequence $f_i(n) = g(\ell n + p(i))$. Then by \cite[Theorem~B]{Leibman05a} again, 
\[
    \overline{\{f_i(n) x: \ n \in \Z\}} = \overline{\{g(\ell n + p(i)) x: \ n \in \Z\}} = Y_{p(i) \bmod \ell}. 
\]
The coefficients of the polynomial $n \mapsto p(\ell n + i) - p(i)$ are divisible by $\ell$ and so the function $n \mapsto \frac{p(\ell n + i) - p(i)}{\ell}$ is a polynomial on $\Z$.
Since $Y_{p(i) \bmod \ell}$ is connected, by \cref{lem:frantzikinakis_connected}, 
\[
    Y_{p(i) \bmod \ell} = \overline{
    \left\{f_i\left(\frac{p(\ell n + i) - p(i)}{\ell} \right)x : \ n \in \Z\right\}} = \overline{\{g(p(\ell n + i))x: \ n \in \Z\}}.
\]
It follows that 
\[
    \overline{\{h(n) x: \ n \in \Z\}} = \bigcup_{i=0}^{\ell - 1} \overline{\{g(p(\ell n + i))x: \ n \in \Z\}} = \bigcup_{i=0}^{\ell-1} Y_{p(i) \bmod \ell}.
\]
In other words, $\orb_h(x)$ consists of some connected components of $\orb_g(x)$.
By \cref{lem:leibman_connected_components_of_translate_subnilmanifold}, the connected components of $\orb_g(V)$ and $\orb_h(V)$ are  translates of $HV$; in particular, they have the same dimension. Since $\orb_g(V)$ is already connected, we deduce that $\orb_g(V)$ itself is a translate of $HV$.  
On the other hand, $\orb_h(V) \subseteq \orb_g(V)$. Therefore, any connected component of $\orb_h(V)$ is a subnilmanifold of the connected nilmanifold $\orb_g(V)$ of the same dimension and so it must be that $\orb_h(V) = \orb_g(V)$. 
\end{proof}

\begin{lemma}
\label{lem:dense_st_same_orbit_closure}
Let $X = G/\Gamma$ be a nilmanifold and $g(n)$ be a polynomial sequence on $G$. Let $Y \subseteq V \subseteq X$ be subsets such that $Y$ is dense in $V$. Then
\[
    \orb_g(Y) = \orb_g(V).
\]
\end{lemma}
\begin{proof}
Let $n \in \Z$ and $x \in V$ be arbitrary. Since $Y$ is dense in $V$, there exists a sequence $(x_i) \subseteq Y$ such that $x_i \to x$. It follows that $g(n) x_i \to g(n) x$ as $i \to \infty$. Thus,
\[
    g(n) x \in \orb_g(Y).
\]
Since $n, x$ are arbitrary, we get $\orb_g(V) \subseteq \orb_g(Y)$.
\end{proof}

\begin{lemma}\label{lem:disconnected_cannot}
Let $X = G/\Gamma$ be a nilmanifold and $V$ be a connected subnilmanifold of $X$. Let $g(n)$ be a polynomial sequence on $G$. If $\orb_{g}(V)$ is disconnected, then there exist $k, j_1, j_2 \in \N$ such that 
\[
    \overline{\{g(kn + j_1) V: \ n \in \Z\}} \cap \overline{\{g(kn + j_2) V: \ n \in \Z\}} = \emptyset.
\]
\end{lemma}

\begin{remark}
If $V$ is a single point, then the lemma follows from \cite[Theorem B]{Leibman05a}. 
\end{remark}

\begin{proof}
By \cite[Theorem 1.13]{Leibman07}, for each $k \in \N$ and $j \in \{0, \ldots, k - 1\}$, the set of points in $V$ which are non-generic for the polynomial sequence $(g(kn + j))_{n \in \Z}$ has zero measure (with respect to the Haar measure on $V$). 
Therefore, there exists $x_0 \in V$ which is generic for $(g(kn + j))_{n \in \Z}$ for all $k \in \N, j \in \{0, \ldots, k - 1\}$. 
By \cite[Theorem B]{Leibman05a}  there exists $k$ such that $\overline{\{g(kn + j) x_0: \ n \in \Z\}}$ is connected for any $j \in \{0, \ldots, k - 1\}$. Since $x_0$ is generic for $(g(k n + j))_{n \in \Z}$ for $j \in \{0, \ldots, k - 1\}$, by \cite[Theorem 1.13]{Leibman07} again, for generic $x \in V$ and $j \in \{0, \ldots, k - 1\}$, $\{g(k n + j) x: \ n \in \Z\}$ is connected.

We claim that for any $j \in \Z$, $\orb_{g(kn + j)}(V) = \overline{\{g(kn + j) V: \ n \in \Z\}}$ is connected. 
For contradiction, assume that there exist disjoint open sets $U_1, U_2 \subseteq X$ both of which intersect $\orb_{g(kn + j)} (V)$ and
\[
    \orb_{g(kn + j)}(V)  \subseteq U_1 \cup U_2.
\]
Since $V$ is connected, $g(k + j) V$  is connected and so it belongs to either $U_1$ or $U_2$. 
Without loss of generality, assume $g(k+j) V \subseteq U_1$. For generic $x \in V$, $g(k+j)(x) \in U_1$ and $\orb_{g(kn + j)} (x)$ is connected. 
As a result, 
\[
    \orb_{g(kn + j)}(x) \subseteq U_1.
\]
Now by \cref{lem:dense_st_same_orbit_closure},
\[
    \overline{\{g(kn + j) V: \ n \in \Z\}} = \bigcup_{x \in V \text{generic}} \overline{\{g(kn + j) x: \ n \in \Z\}}.
\]
Therefore, 
\[
    \orb_{g(kn + j)}(V)  \subseteq U_1  
\]
contradicting the assumption that $ \orb_{g(kn + j)}(V) \cap U_2 \neq \emptyset$. Thus $\orb_{g(kn + j)}(V)$ is connected for all $j \in \Z$.

If 
\[
    \orb_{g(kn + j_1)} (V) \cap \orb_{g(kn + j_2)} (V) \neq \emptyset
\]
for all $j_1, j_2$, then all $\orb_{g(kn + j)}(V)$ ($j \in \Z$) belong to a single connected component of $\orb_g(V)$. But this is impossible since $\orb_g(V)$ is disconnected and
\[
    \orb_g(V) = \bigcup_{j=0}^{k-1} \orb_{g(kn + j)}(V).
\]
The proof is now complete.
\end{proof}

\subsubsection{\texorpdfstring{\cref{conj:leibmanish}}{Conjecture 1.4} along progressions}

We begin by proving that \cref{conj:leibmanish} is equivalent to the following version along arithmetic progressions.

\begin{Conjecture}
\label{lem:connected_pol_first}
Let $p_1, \ldots, p_d \in \Z[x]$ with $p_i(0) = 0$ for all $1 \le i \le d$, and let $(X,T)$ be a totally minimal nilsystem with $X = G/\Gamma$ and $Tx = ax$. If $g(n) = (a^{p_1(n)},\dots,a^{p_d(n)}),$ then for all $k \in \N$ and all $j \in \Z$,
\[
\Delta_{X^d} \subseteq \overline{\{g(kn+j) \Delta_{X^{d}}: \ n \in \Z\}}.
\]
\end{Conjecture}

\begin{theorem}
\label{thm:leibman_nil_equiv}
\cref{conj:leibmanish} and \cref{lem:connected_pol_first} are equivalent.
\end{theorem}

\begin{proof}
Fix polynomials $p_1,\dots,p_d \in \Z[x]$ with $p_i(0) = 0$ for all $1 \le i \le d$ and a totally minimal nilsystem $(X,T)$ with $X = G/\Gamma$ and $Tx = ax$. Write $g(n) = (a^{p_1(n)},\dots,a^{p_d(n)})$.

(\cref{lem:connected_pol_first} implies \cref{conj:leibmanish}) \ 
Suppose $\overline{\{g(n) \Delta_{X^d}: \ n \in \Z\}}$ is not connected. 
By \cref{lem:disconnected_cannot}, there exist $k, j_1 \neq j_2$ such that $\overline{\{g(kn + j_1) \Delta_{X^d}: \ n \in \Z\}}$ and $\overline{\{g(kn + j_2) \Delta_{X^d}: \ n \in \Z\}}$ are disjoint. However, by \cref{lem:connected_pol_first}, the diagonal belongs to both sets which is a contradiction.

(\cref{conj:leibmanish} implies \cref{lem:connected_pol_first}) \ 
By \cite{Bergelson_Leibman_Lesigne08} we have $\overline{x} \in \orb_g(\overline{x})$ for every $\overline{x} \in \Delta_{X^d}$.
As a result,
\[
\Delta_{X^d} \subseteq \orb_g \Delta_{X^d}.
\]
Let $k \in \N$, $j \in \Z$ be arbitrary and let $h(n) = g(kn + j)$. By \cref{conj:leibmanish}, $\orb_g(\Delta_{X^d})$ is connected. By \cref{prop:orbit_closure_subnil_arith_same}, $\orb_h(\Delta_{X^d}) = \orb_g(\Delta_{X^d})$ and so $\orb_h(\Delta_{X^d})$ contains $\Delta_{X^d}$ as desired.
\end{proof}

\subsubsection{Proof of \texorpdfstring{\cref{mainthm_conjecture_are_equiv}}{Theorem E}}

\label{sec:all_equivalent_conjectures}

By \cref{thm:leibman_nil_equiv} it remains to prove \cref{conj:polynomial-odd-recurrence} is equivalent to \cref{lem:connected_pol_first}. We will prove that both are equivalent to the following conjectures.

\begin{Conjecture}[Syndetic version of \cref{conj:polynomial-odd-recurrence}]
\label{conj:polynomial-odd-recurrence-syndetic}
Fix $k\in \N$ and non-constant polynomials $p_1,\dots,p_d \in \Z[x]$ with $p_i(0) = 0$ for all $1 \le i \le d$. If $(X,T^k)$ is minimal, then for all open $\emptyset\neq U\subseteq X$ and all $0 \leq j < k$, the set
\[
\{ n \in \Z : \ U \cap T^{-p_1(kn+j)} U \cap \cdots \cap T^{-p_d(kn+j)} U \ne \emptyset \}
\]
is syndetic.
\end{Conjecture}

\begin{Conjecture}[\cref{conj:polynomial-odd-recurrence} for totally minimal nilsystems]
\label{conj:glasner-etal-totally-minimal}
If $(X,T)$ is totally minimal nilsystem, then for all polynomials $p_1,\dots,p_d$ on $\Z$ with $p_i(0)=0$ for all $1\leq i\leq d$, all $k\in \N,$ all open $\emptyset\neq U\subseteq X$, and all $0 \leq j < k$, there exists $n\in\N$ such that
\[
U \cap T^{-p_1(kn+j)}U\cap\cdots\cap T^{-p_d(kn+j)}U\neq\emptyset.
\]
\end{Conjecture}

We will show that all the conjectures above are equivalent. Before going to the proof, we need some lemmas.

\begin{lemma}\label{lem:change_polynomial}
    Let $p_1,\dots,p_d \in \Z[x]$ with $p_i(0)=0$ for all $1\leq i\leq d$, $M\in\N$, $a,b\in\Z$ with $a\neq 0$, and $q_i(n) \defeq p_i\big(M(an+b)\big)/M$ for all $1\leq i\leq d.$
    For all $1\leq i\leq d,$ there exist polynomials $\tilde p_1, \ldots, \tilde p_d \in \Z[x]$ with $\tilde p_i(0)=0$ such that for all $n \in \Z$,
    \[q_i(n)=\tilde p_i(an+b).\]
\end{lemma}

\begin{proof}
    Let $d\in\N$ be the largest degree of the $p_i$'s, so we can write $p_i(n)=\sum_{\ell=1}^dc_{i,\ell}n^\ell$.
    Then 
    $$q_i(n)
    =
    \frac1M\sum_{\ell=1}^dc_{i,\ell}(M(an+b))^\ell
    =
    \sum_{\ell=1}^dc_{i,\ell}M^{\ell-1}(an+b)^\ell.$$
    Letting $\tilde p_i(n)=\sum_{\ell=1}^dc_{i,\ell}M^{\ell-1}n^\ell$ we have the conclusion.
\end{proof}

\begin{Proposition}
\label{thm:main_nil_pol}
Suppose \cref{lem:connected_pol_first} is true. Let $(X,T)$ be a minimal nilsystem with $X = G/\Gamma$ and $Tx = ax$. 
Assume that  $(X,T^k)$ minimal. 
Let $p_1, \ldots, p_d$ be non-constant polynomials on $\Z$ satisfying $p_i(0) = 0,$ for $1\leq i\leq d,$ and let $g(n) = 1_G \otimes a^{p_1(n)} \otimes \cdots \otimes a^{p_d(n)}$ be the corresponding polynomial sequence in $G$. 
If $\Delta_{X^{d+1}}$ denotes the diagonal of $X^{d+1}$, then, for every $j \in \Z/k\Z$,

\[
    \Delta_{X^{d+1}} \subseteq \overline{\{g(kn + j) \Delta_{X^{d+1}}: \ n \in \Z\}}.
\]
\end{Proposition}
\begin{proof}
Let $M$ be the number of connected components of $X$ and let $X_0$ be the component containing $1_X$. 
Then $(X_0, T^M)$ is totally minimal. 
Since $(X,T^k)$ is minimal, it follows that $\gcd(M, k) = 1$. 
It follows that there exists $t \in \Z/M\Z$ such that $kt + j \equiv 0 \bmod M$.

For $1\leq i \leq d$, consider the polynomials
\[
    q_i(n) = \frac{p_i(k(Mn + t) + j)}{M} = \frac{p_i(kMn + kt + j)}{M}.
\]
Since $kt + j \equiv 0 \bmod M$, by \cref{lem:change_polynomial}, there exist polynomials $\tilde{p_i}$ over $\Z$ with $\tilde{p_i}(0) = 0$ such that 
\[
    q_i(n) = \tilde{p_i}(kn + b),
\]
where $b = (kt + j)/M$.

Now letting $g=a^M,$ we have that $(X_0, g)$ is a totally minimal nilsystem and 
\[
    h(n) = 1_G \otimes g^{q_1(n)} \otimes \cdots \otimes g^{q_d(n)} = 1_G \otimes g^{\tilde{p_1}(kn + b)} \otimes \cdots \otimes g^{\tilde{p_d}(kn + b)}
\]
is a polynomial sequence satisfying the assumptions of \cref{lem:connected_pol_first} (because $\tilde p_i(0)=0$).
We deduce that
\[
    \Delta_{X_0^{d+1}} \subseteq \overline{\{h(n) \Delta_{X_0^{d+1}}: \ n \in \Z\}}.
\]
Since $\Delta_{X^{d+1}} =\bigcup_{\ell=0}^{M-1}(T\times\cdots\times T)^\ell\Delta_{X_0^{d+1}}$, 
\[
    \overline{\{h(n) \Delta_{X_0^{d+1}}: \ n \in \Z\}} \subseteq \overline{\{g(kn + j) \Delta_{X^{d+1}}: \ n \in \Z\}},
\]
and the $\{g(kn + j) \Delta_{X^{d+1}}: \ n \in \Z\}$ is invariant under $T\times\cdots\times T$, the conclusion follows.
\end{proof}

\cref{mainthm_conjecture_are_equiv} from the introduction is a consequence of the following theorem.

\begin{theorem}
\label{thm:all_conjectures_equivalent}
Conjectures~\ref{conj:polynomial-odd-recurrence}, \ref{conj:leibmanish}, \ref{lem:connected_pol_first}, \ref{conj:polynomial-odd-recurrence-syndetic}, and \ref{conj:glasner-etal-totally-minimal} are all equivalent.
\end{theorem}

\begin{proof}
The implications \cref{conj:polynomial-odd-recurrence-syndetic} $\Rightarrow$ \cref{conj:polynomial-odd-recurrence} $\Rightarrow$ \cref{conj:glasner-etal-totally-minimal} are immediate, and Conjectures \ref{conj:leibmanish} and \ref{lem:connected_pol_first} were already shown to be equivalent in \cref{thm:leibman_nil_equiv}. We will finish the proof by showing that \cref{conj:glasner-etal-totally-minimal} implies \cref{lem:connected_pol_first} and that \cref{lem:connected_pol_first} implies \cref{conj:polynomial-odd-recurrence-syndetic}.\\

(\cref{conj:glasner-etal-totally-minimal} $\Rightarrow$ \cref{lem:connected_pol_first}) \ 
Fix a totally minimal nilsystem $(X, T)$ with $X = G/\Gamma$ and $T x = ax$ and polynomials $p_1,\dots,p_d$ over $\Z$ with $p_i(0) = 0$ for all $1 \le i \le d$, and write $g(n) = (a^{p_1(n)},\dots,a^{p_d(n)})$.
Fix $k \in \N$ and $0 \le j < k$.
For each $\ell \in \N$ we break $X$ into finitely many open balls of radius $\epsilon < \tfrac{1}{\ell}$ and conclude from \cref{conj:glasner-etal-totally-minimal} that
\[
R_\ell := \Big\{x\in X: \ (\exists n\in\N)\ d\big((x,\dots,x), g(kn+j)(x,\dots,x)\big)<1/\ell\Big\}
\]
is open and $\epsilon$-dense.
Since $\epsilon$ is arbitrary, $R_\ell$ is an open dense set.
By Baire category, the intersection $R:=\bigcap_\ell R_\ell$ is dense.
Fix $y \in X$ and let $(x_i)_{i \in \N}$ be a sequence in $R$ converging to $y$.
For each $i\in\N$ let $n_i\in\N$ be such that $d((x_i,\dots,x_i),g(kn_i + j)(x_i,\dots,x_i))<1/i$. 
It follows that
\[
g(kn_i + j)(x_i,\dots,x_i) \to (y,\dots,y)
\]
as $i\to\infty$. As $y \in X$ was arbitrary, we have the conclusion of \cref{lem:connected_pol_first}.\\

(\cref{lem:connected_pol_first} $\Rightarrow$ \cref{conj:polynomial-odd-recurrence-syndetic}) \ 
Fix $(X,T)$ minimal and fix $k \in \N$ with $(X,T^k)$ minimal.
Fix also $\emptyset \ne U \subseteq X$ open and $0 \le j < k$.
Put $q_i(n) = p_i(kn+j)$ for all $1 \le i \le d$ and write $q = (q_1,\dots,q_d)$.
We need to prove that $R_q(U,\dots,U)$ is syndetic.

Let $(X_{\infty}, T)$ be the $\infty$-step pronilfactor of $(X, T)$ and write $\pi: X \to X_{\infty}$ for the factor map.
From \cref{mainthm_ps_structure_theory_pol} we conclude that
\[
    R_q((\pi U)^\circ, \ldots, (\pi U)^\circ) \setminus R_q(U, \ldots, U) 
\]
is not piecewise syndetic.
From \cref{lemma_pwssetdifference} it now suffices to prove $R_q((\pi U)^\circ, \ldots, (\pi U)^\circ)$ is syndetic.

The map $\pi$ is semi-open and so $\pi U$ has a nonempty interior; in particular $(\pi U)^\circ$ is a nonempty, open subset of $X_{\infty}$. The space $X_{\infty}$ is an inverse limit of nilsystems so (cf.\ \cref{rmk_open_set_in_an_inverse_limit}) there is an nilsystem factor of $X_{\infty}$, say $(Y, T)$, and a nonempty, open $V \subseteq Y$ such that $\phi^{-1}(V) \subseteq (\pi U)^\circ$ where $\phi: X_{\infty} \to Y$ is the factor map.
Thus $R_q(V, \ldots, V) \subseteq R_q((\pi U)^\circ, \ldots, (\pi U)^\circ)$ and it suffices to prove $R_q(V, \ldots, V)$ is syndetic.

The system $(Z,T^k)$ is minimal as a factor of $(X,T^k)$.
\cref{thm:main_nil_pol} implies
\[
\Delta_{Z^d} \subseteq \overline{\{g(n) \Delta_{Z^d}: \ n \in \Z\}},
\]
where  $T(z) = az$ and $g(n) = (a^{q_1(n)},\dots,a^{q_d(n)})$.
Since
\[
\Delta_{Z^d} \cap V^d = \{(z, \ldots, z): \ z \in V\} \neq \emptyset
\]
it follows that $\overline{\{g(n) \Delta_{Z^d}: \ n \in \Z\}} \cap V^d$ is a nonempty, open subset of $\overline{\{g(n) \Delta_{Z^d}: \ n \in \Z\}}$. Therefore, by \cite[Corollary 1.9]{Leibman05b}, the set
\[
    S = \{n \in \Z: \ \Delta_{X^d} \cap g(n)^{-1} V^d \neq \emptyset\} = \{n \in \Z: \ g(n) \Delta_{Z^d} \cap V^d \neq \emptyset\} 
\]
is syndetic. This set $S$ is exactly $R_q(V, \ldots, V)$ and so we are done.
\end{proof}

\section{Further discussion and open questions}
\label{sec_further}

In this section, we use the assumptions and conclusions in \cref{mainthm_ps_structure_theory_pol} to generate further discussion and open questions.  We address the role of essential distinctness and the $\infty$-step pronilfactor (\cref{sec_equicont_factor}); the role of the set interiors (\cref{sec_interiors_role}); the invertibility of $(X,T)$ (\cref{sec_non_invert}); piecewise syndeticity as a measure of largeness (\cref{sec_density_vs_ps}); and the possibility of an ergodic-theoretic analogue (\cref{sec_ergodic_analogue}).

\subsection{Essential distinctness and the infinite step pronilfactor}
\label{sec_equicont_factor}

The essential distinctness assumption on $p \in \Z[x]^d$ in Theorems \ref{thm_qiu_eqiuvalent_statement} and \ref{mainthm_ps_structure_theory_pol} is necessary. Indeed, suppose $d=2$ and $p_2 - p_1 = m$ for some $m\in\Z$. Let $(X, T)$ be a topologically weakly mixing system and $U_1, U_2 \subseteq X$ be nonempty, open sets such that $U_1 \cap T^{-m} U_2 = \emptyset$. We see that
\[
    R_p(U_1, U_2) = \{n \in \Z: \ T^{-p_1(n)} U_1 \cap T^{-p_2(n)} U_2 \neq \emptyset\} = \{n \in \Z: \ U_1 \cap T^{-(p_2(n) - p_1(n))} U_2 \neq \emptyset\} = \emptyset.
\]
On the other hand, since $(X,T)$ is weakly mixing, the factor $(X_{\infty},T)$ is the trivial system, so $\pi U_1=\pi U_2=X_\infty$ and $R_p((\pi U_1)^{\circ}, \allowbreak (\pi U_2)^{\circ}) = \Z$.

The example in the following proposition shows that \cref{mainthm_ps_structure_theory_pol} is no longer true if the $\infty$-step pronilfactor of $(X,T)$ is replaced by the system's maximal equicontinuous factor (recall the definition from \cref{sec_prelims_top_dyn}).

\begin{proposition}
\label{prop_max_equicont_factor_does_not_suffice}
There exists a minimal $2$-step nilsystem $(X, T)$ and a nonempty, open set $U \subseteq X$ such that if $\pi: X \to Z$ is the factor map to the maximal equicontinuous factor, then
\[
    R((\pi U)^\circ, (\pi U)^\circ, (\pi U)^\circ) \setminus R(U, U, U)
\]
is syndetic.
\end{proposition}
\begin{proof}
For $x \in \T = \R/\Z$, let $\lVert x \rVert$ denote the distance from $x$ to $0$.
Let $X = \T \times \T$ and $T: X \to X$ be defined by $T(x, y) = (x + \alpha, y + x)$ where $\alpha$ is an irrational number. Then
\[
    T^n(x, y) = \left(x + n \alpha, y + n x + \frac{n(n-1)}{2} \alpha\right).
\]

Fix $\epsilon > 0$ and define $U \subseteq X$ to be the ball centered at $(0, 0)$ having radius $\epsilon/4$. 
Suppose $n \in R(U, U, U)$ and let $(x_0, y_0) \in U \cap T^{-n} U \cap T^{-2n} U$. Then
\begin{align*}
        \max\{\lVert x_0 \rVert, \lVert y_0 \rVert\} &< \epsilon/4, \\
    \max\left\{\lVert x_0 + n \alpha \rVert, \left\lVert y_0 + n x_0 + \frac{n(n-1)}{2} \alpha \right\rVert\right\} &< \epsilon/4, \text{ and} \\
    \max\left\{\lVert x_0 + 2n \alpha \rVert,
    \left\lVert y_0 + 2n x_0 + \frac{2n(2n-1)}{2} \alpha \right\rVert\right\} &< \epsilon/4.
\end{align*}
It follows that
\[
    \lVert n^2 \alpha \rVert = \left\lVert \left(y_0 + 2n x_0 + \frac{2n(2n-1)}{2} \alpha\right) - 2\left(y_0 + n x_0 + \frac{n(n-1)}{2} \alpha\right) + y_0\right\rVert < \epsilon.
\]
Therefore,
\[
    R(U, U, U) \subseteq Q := \{n \in \Z: \ \lVert n^2 \alpha \rVert < \epsilon\}.
\]

The maximal equicontinuous factor of $(X, T)$ is the rotation by $\alpha$ on the first coordinate and so $\pi(U)^\circ = \pi(U)= \{x \in \T: \ \lVert x \rVert < \epsilon/4\}$. Let 
\[
    B = \{n \in \Z: \ \lVert n \alpha \rVert < \epsilon /8\}.
\]
Then for $n \in B$,
\[
    \max\{\lVert 0 \rVert, \lVert 0 + n \alpha \rVert, \lVert 0 + 2 n \alpha \rVert\} < \epsilon/4.
\]
As a result, $n \in R((\pi U)^\circ, (\pi U)^\circ, (\pi U)^\circ)$ and since $n$ is arbitrary, $B \subseteq R((\pi U)^\circ, (\pi U)^\circ, (\pi U)^\circ)$.

To show $R((\pi U)^\circ, (\pi U)^\circ, (\pi U)^\circ) \setminus R(U, U, U)$ is syndetic, it remains to show that $B \setminus Q$ is syndetic. Note that 
\[
    B \setminus Q = \{n \in \Z: \ \lVert n \alpha \rVert < \epsilon/8, \lVert n^2 \alpha \rVert \geq \epsilon\}.
\]
Consider the system $(X, S)$ with 
\[
    S(x, y) = (x + \alpha, y + 2x + \alpha)
\]
and the set
\[
    V = \{(x, y) \in X: \lVert x \rVert < \epsilon/8, \lVert y \rVert > \epsilon\}.
\]
Then $V$ is a nonempty open subset of $X$. We have
\[
    S^n(0, 0) = (n \alpha, n^2 \alpha)
\]
and therefore
\[
    B \setminus Q \supseteq \{n \in \Z: S^n (0, 0) \in V\}.
\]
Since the system $(X, S)$ is minimal, $B \setminus Q$ is syndetic. 
\end{proof}

While the maximal equicontinuous factor does not suffice to reach the conclusions in \cref{mainthm_ps_structure_theory_pol}, it is natural to speculate -- especially in light of Ye and Yu's result \cite[Theorem~A]{yeyu} -- that a pronilfactor of sufficiently high-order does.

\begin{question}
\label{quest_suff_high_order_factor_suffices}
    Given $d \in \N$ and an essentially distinct polynomial tuple $p \in \Z[x]^d$, is there $k \in \N$ such that for all minimal, invertible systems $(X,T)$ and all nonempty, open $U_1, \ldots, U_d \subseteq X$, the set
    \[
        R_p\big((\pi U_1)^\circ, \ldots, (\pi U_d)^\circ\big) \setminus R_p(U_1, \ldots, U_d)
    \]
    is not piecewise syndetic, where $\pi: X \to X_k$ denotes the factor map to the maximal $k$-step pronilfactor of $(X,T)$?
\end{question}

The methods we employ in the proof of \cref{mainthm_ps_structure_theory_pol} do not seem to be easily modified to give a positive answer to \cref{quest_suff_high_order_factor_suffices}.  Specifically, it is in \cref{lemma_special_sets_of_returns_form_a_filterpol} that we encounter arbitrarily high-order return-time sets that then necessitate arbitrarily high-order pronilfactors.

\subsection{The role of the set interiors}
\label{sec_interiors_role}

The set interiors $(\pi U_i)^\circ$ appearing in the statement of \cref{mainthm_ps_structure_theory_pol} originate from the almost 1--1 extensions necessary to extend a semiopen factor map to an open one; see, specifically, \cref{lemma_multiple_set_returns_lift_in_almost_11_extensions}.  It remains unclear to us whether or not their appearance is necessary. The following question asks for a strengthening of \cref{mainthm_ps_structure_theory_pol} in which each set $(\pi U_i)^\circ$ is replaced by the larger set $\pi U_i$.  The sets $\pi U_i$ are generally not open, but we use the notation set out in \eqref{eqn_main_rp_notation} nonetheless.

\begin{question}
\label{quest_remove_interiors}
    Let $(X,T)$ be a minimal and invertible topological dynamical system. Denote by $(X_\infty,T)$ its maximal $\infty$-step pro-nilfactor, and let $\pi\colon X\to X_\infty$ be the associated factor map.
    For all nonempty, open $U_1, \ldots, U_d \subseteq X$ and all essentially distinct $p \in \mathbb{Z}[x]^d$, is it true that the set
    \[
        R_p\big(\pi U_1, \ldots, \pi U_d\big) \setminus R_p(U_1, \ldots, U_d)
    \]
    is not piecewise syndetic?
\end{question}

In attempting to answer \cref{quest_remove_interiors} using \cref{mainthm_ps_structure_theory_pol}, one is naturally led to consider the inclusion
\[R_p\big(\pi U_1, \ldots, \pi U_d\big) \supseteq R_p\big((\pi U_1)^\circ, \ldots, (\pi U_d)^\circ\big)\]
and ask how much larger the set $R_p\big(\pi U_1, \ldots, \pi U_d\big)$ might be.  Since the factor map from a minimal system to its $\infty$-step pronilfactor is semiopen but generally not open, the inclusion $(\pi U_i)^\circ \subseteq \pi U_i$ is generally strict.
Nevertheless, it can be shown that $(\pi U_i)^\circ$ is always a dense subset of $\pi U_i$, that is, $\pi U_i \subseteq \overline{(\pi U_i)^\circ}$. Thus, a natural approach to answering \cref{quest_remove_interiors} is to consider the difference between the times of visits of an array of open sets and their closures.  It can be shown that a positive answer to the following question would yield a positive answer to \cref{quest_remove_interiors}.

\begin{question}
\label{quest_orbits_of_diagonal_along_thick_sets_in_nilsystems_1}
Let $(X,T)$ be a minimal pronilsystem and $d\in\N$.
For all nonempty, open $U_1, \ldots, U_d \subseteq X$ and all essentially distinct $p \in \Z[x]^d$, is it true that the set
\[
R_p\big(\overline{U_1}, \ldots, \overline{U_d}\big) \setminus R_p(U_1, \ldots, U_d  )
\]
is not piecewise syndetic?
\end{question}

As far as we know, it is possible that \cref{quest_orbits_of_diagonal_along_thick_sets_in_nilsystems_1} has a positive answer in every minimal system.

\subsection{Results in the non-invertible setting}
\label{sec_non_invert}

It is natural to wonder how important the invertibility assumption is on the system $(X,T)$ in \cref{mainthm_ps_structure_theory_pol}.  One issue arises immediately: to our knowledge, there is no recorded definition of the $\infty$-step pronilfactor of a non-invertible system. More to the point, it appears that there has been no attempt in the literature to define the regionally proximal relations for $\N$-actions.

We will assume for the remainder of this subsection that a reasonable definition has been made for the $\infty$-step pronilfactor of a not-necessarily-invertible minimal system.  Given such a system $(X,T)$, nonempty, open sets $U_1,\ldots,U_d \subseteq X$, and a polynomial tuple $p = (p_1, \ldots, p_d) \in \Z[x]^d$ with the property that each $p_i$ is eventually positive, we define
\begin{align*}
    R^+_p(U_1,\dots,U_d):=\big\{n\in\N : \ T^{-p_1(n)} U_1\cap\cdots\cap T^{-p_d(n)}U_d\neq\emptyset \big\}.
\end{align*}
The following question asks for a non-invertible analogue of \cref{mainthm_ps_structure_theory_pol}.

\begin{question}
\label{quest_non_invert_thm_A}
    Let $(X,T)$ be a minimal, not-necessarily-invertible topological dynamical system. Denote by $(X_\infty,T)$ its maximal $\infty$-step pro-nilfactor, and let $\pi\colon X\to X_\infty$ be the associated factor map.
    For all $d \in \N$, all nonempty, open $U_1, \ldots, U_d \subseteq X$, and all essentially distinct $p \in \mathbb{Z}[x]^d$ with each $p_i$ eventually positive, is it true that the difference set
    \[
        R^+_p\big((\pi U_1)^\circ, \ldots, (\pi U_d)^\circ\big) \setminus R^+_p(U_1, \ldots, U_d)
    \]
    is not piecewise syndetic?
\end{question}

To give a positive answer to \cref{quest_non_invert_thm_A}, it would be natural to follow the proof of \cref{mainthm_ps_structure_theory_pol}.  Except for Qiu's theorem (recall \cref{thm_qiu_eqiuvalent_statement}), which is not known for non-invertible systems, it appears that all of the auxiliary results in \cref{sec_poly_ps_structure} necessary for the proof can be shown to hold in the non-invertible setting.

\subsection{Positive upper Banach density and piecewise syndeticity}
\label{sec_density_vs_ps}

The {\it upper Banach density} of a set $A \subseteq \Z$ is
\[d^*(A) = \limsup_{N \to \infty} \max_{n \in \Z} \frac{|A \cap \{n + 1, \ldots, n + N\}|}{N}.\]
Piecewise syndetic sets have positive upper Banach density; the set of squarefree integers demonstrates that the converse is false.  Given upper Banach density's importance as a tool in the development of ergodic Ramsey theory, it is natural to ask whether the role of piecewise syndeticity as a notion of largeness can be replaced by positive upper Banach density in \cref{mainthm_ps_structure_theory_pol}.

\begin{question}
\label{ques:non-ps-by-zero-density}
Does \cref{mainthm_ps_structure_theory_pol} hold if, in the conclusion, ``is not piecewise syndetic'' is replaced by ``has zero upper Banach density''?
\end{question}

We believe the answer to \cref{ques:non-ps-by-zero-density} is negative, perhaps already in the case of single recurrence.

\begin{question}\label{ques:sepecial_case}
Does there exist a minimal topological system $(X, T)$, nonempty, open sets $U_1, U_2 \subseteq X$, and a non-constant polynomial $p \in \Z[x]$ such that the set
\[
\big\{ n \in \Z : U_1 \cap T^{-p(n)} U_2 = \emptyset \text{ and } (\pi U_1)^\circ \cap T^{-p(n)} (\pi U_2)^\circ \neq \emptyset \big\} 
\]
has positive upper Banach density? Here $\pi: X \to X_{\infty}$ is the projection to the topological $\infty$-pronilfactor of $(X,T)$.
\end{question}

\subsection{An ergodic-theoretic analogue of \texorpdfstring{\cref{mainthm_ps_structure_theory_pol}}{Theorem A}}
\label{sec_ergodic_analogue}

The use of structured factors of topological dynamical systems in the study topological multiple recurrence, initiated in~\cite{MR1303514}, has seen significant development in recent years~\cite{glasner_huang_shao_weiss_ye_2020} motivated by development of the structure theory for measure-preserving systems by Host-Kra and Ziegler in \cite{Host_Kra05,Ziegler07}, and many parallels exist between the two frameworks. It is therefore natural to ask whether our main result, \cref{mainthm_ps_structure_theory_pol}, which characterizes the structure of return-time sets in minimal topological dynamical systems, admits a measure-theoretic analogue describing in a similar fashion the structure of return-time sets in ergodic measure-preserving systems.
In this section, we explore this topic in more detail, formulating a concrete question that can be viewed as an ergodic counterpart to \cref{mainthm_ps_structure_theory_pol}, and presenting a partial result in support of it.

Up to this point in the paper, we have worked with topological dynamical systems $(X, T)$ and have considered their topological $\infty$-step pronilfactors, denoted by $(X_\infty, T)$.
Since we now shift our focus to the ergodic-theoretic setting, we consider instead measure-preserving systems $(X,\mu,T)$, where $X$ is a compact metric space, $\mu$ is a Borel probability measure on $X$, and $T\colon X\to X$ is a measurable transformation that preserves the measure $\mu$, together with their measure-theoretic $\infty$-step pronilfactors $(Z_\infty,\mu,T)$. For a definition of the measure-theoretic $\infty$-step pronilfactors, we refer to~\cite{Host_Kra18}.
We remark that when $(X,T)$ is a minimal topological system, its topological $\infty$-step pronilfactor $(X_\infty,T)$ is a factor of its measure-theoretic $\infty$-step pronilfactor $(Z_\infty,\mu,T)$ for any $T$-invariant and ergodic probability measure $\mu$ on $(X,T)$, but in general the latter can be larger.

\begin{theorem}
\label{thm_imperfect_ergodic_analogue_of_thm_a}
Let $(X, \mu, T)$ be an ergodic measure preserving system and $(Z_{\infty}, \mu, T)$ be its measure-theoretic $\infty$-step pronilfactor. Let $p = (p_1, \ldots, p_d) \in \Z[x]^d$ be essentially distinct and $f_1, \ldots, f_d: X \to [0,\infty)$ bounded measurable functions. For all $\eps_1<\eps_2\in (0,\infty)$, the sets
\[
    R_{X,\eps_1} = \bigg\{n \in \Z: \int_X T^{p_1(n)} f_1 \cdots T^{p_d(n)} f_d \ d \mu > \eps_1\bigg\}
\]
and
\[
    R_{Z_{\infty}, \eps_2} = \bigg\{n \in \Z: \int_X T^{p_1(n)} \E(f_1|Z_{\infty}) \cdots T^{p_d(n)} \E(f_d|Z_{\infty}) \ d \mu > \eps_2\bigg\}
\]
satisfy 
\begin{align}
\label{eqn_diff_of_msble_returns}
    d^*(R_{Z_{\infty}, \eps_2} \setminus R_{X,\eps_1}) = 0.
\end{align}
\end{theorem}

\begin{proof}
Define the sequences
\[
    \alpha(n) = \int_X T^{p_1(n)} f_1 \cdots T^{p_d(n)} f_d \ d \mu
\]
and
\[
    \beta(n) = \int_X T^{p_1(n)} \E(f_1|Z_{\infty}) \cdots T^{p_d(n)} \E(f_d|Z_{\infty}) \ d \mu.
\]
By combining ideas from \cite{Leibman05} and \cite{Bergelson_Host_Kra05}, one can show that
\[
    \lim_{N - M \to \infty} \frac{1}{N - M} \sum_{n=M}^N |\alpha(n) - \beta(n)| = 0.
\]
Now we have
\[
    R_{Z_{\infty}, \eps_2} \setminus R_{X,\eps_1} \subseteq \{n \in \Z: |\alpha(n) - \beta(n)| > \eps_2-\eps_1\}
\]
and so has Banach density zero.
\end{proof}

Does the conclusion in \eqref{eqn_diff_of_msble_returns} of \cref{thm_imperfect_ergodic_analogue_of_thm_a} hold when $\eps_1 =\eps_2 = 0$?  If true, this would yield a tighter ergodic-theoretic analogue to \cref{mainthm_ps_structure_theory_pol}.  The argument given above fails when $\eps_1 =\eps_2 = 0$, but we could not rule out the possibility that the statement is true.

\begin{question}
\label{quest_ergodic_analogue}
    Let $(X, \mu, T)$ be an ergodic measure preserving system and $(Z_{\infty}, \mu, T)$ be its $\infty$-step pronilfactor. Let $p = (p_1, \ldots, p_d) \in \Z[x]^d$ be essentially distinct. Let $f_1, \ldots, f_d: X \to [0,\infty)$ be bounded measurable functions. Do the sets
\[
    R_{X,0} = \bigg\{n \in \Z: \int_X T^{p_1(n)} f_1 \cdots T^{p_d(n)} f_d \ d \mu > 0\bigg\}
\]
and
\[
    R_{Z_{\infty}, 0} = \bigg\{n \in \Z: \int_X T^{p_1(n)} \E(f_1|Z_{\infty}) \cdots T^{p_d(n)} \E(f_d|Z_{\infty}) \ d \mu > 0\bigg\}
\]
satisfy 
\[
    d^*(R_{Z_{\infty}, 0} \setminus R_{X,0}) = 0?
\]
\end{question}

A first step toward answering \cref{quest_ergodic_analogue} would be to do so in the case that $d=2$, $p_1(n) = 0$, and $p_2(n) = n$.

\bibliographystyle{abbrv}
\bibliography{bib}
\end{document}